\documentclass[final,hidelinks,onefignum,onetabnum]{siamart220329}
\usepackage{lipsum}
\usepackage{amsfonts}
\usepackage{graphicx}
\usepackage{epstopdf}
\usepackage{algorithmic}
\ifpdf
\DeclareGraphicsExtensions{.eps,.pdf,.png,.jpg}
\else
\DeclareGraphicsExtensions{.eps}
\fi

\usepackage{booktabs}
\usepackage{hyperref}
\usepackage{stmaryrd}
\usepackage{bm}
\usepackage{caption}
\usepackage{subcaption}
\usepackage{multirow}
\usepackage{enumitem}
\usepackage{color}

\usepackage{amsopn}

\usepackage{booktabs}
\usepackage[left=2.48cm,right=5.96cm,top=3.28cm,bottom=3.26cm,asymmetric]{geometry}

\newsiamthm{assumption}{Assumption}
\newcommand{\nm}[1]{\left\|#1\right\|}

\def\jump#1{\llbracket #1 \rrbracket }
\newcommand{\mykappa}{\kappa}
\newcommand{\ip}[2]{\left(#1,#2\right)}
\newcommand{\ipi}[2]{\left(#1,#2\right)_i}
\newcommand{\ipjump}[2]{\left[#1,#2\right]}

\newcommand{\sv}[2]{\mathcal{H}^*\left(#1,#2\right)}
\newcommand{\svtd}{\mathbb{D}^*}
\newcommand{\svtde}{\left( \mathbb{D}^* \right)}

\newcommand{\cF}{\mathcal{F}}
\newcommand{\cZ}{\mathcal{Z}}

\newcommand{\quand}{\quad\text{and}\quad}
\newcommand{\hf}{\frac{1}{2}}

\newsiamremark{remark}{Remark}
\newsiamremark{hypothesis}{Hypothesis}
\crefname{hypothesis}{Hypothesis}{Hypotheses}
\newsiamthm{claim}{Claim}
\newsiamremark{exmp}{Example}

\headers{}{}

\title{Spectral Volume from a DG perspective: Oscillation Elimination, Stability, and Optimal Error Estimates 
	\thanks{%Submitted to the editors DATE.
			The authors were partially supported by Shenzhen Science and Technology Program (No.~RCJC20221008092757098) and 
			National Natural Science Foundation of China (No.~12171227).}}

\author{Zhuoyun Li\thanks{Department of Mathematics, Southern University of Science and Technology, Shenzhen, Guangdong 518055, China (\email{12231266@mail.sustech.edu.cn}).}
	\and Kailiang Wu\thanks{Corresponding author. Department of Mathematics and Shenzhen International Center for Mathematics, Southern University of Science and Technology, Shenzhen, Guangdong 518055, China (\email{wukl@sustech.edu.cn}).}} %\footnotemark[3]

\usepackage{amsopn}

\ifpdf
\hypersetup{
  pdftitle={},
  pdfauthor={}
}
\fi

\usepackage{comment}

\newcommand{\nn}{{n_{**}}}

\begin{document}

\maketitle

% REQUIRED
\begin{abstract}
The discontinuous Galerkin (DG) method and the spectral volume (SV) method are two widely-used numerical methodologies for solving hyperbolic conservation laws. In this paper, we demonstrate that under specific subdivision assumptions, the SV method can be represented in a DG form with a different inner product. Building on this  insight, we extend the oscillation-eliminating (OE) technique, recently proposed in [M.~Peng, Z.~Sun, and K.~Wu, {\em Mathematics of Computation}, in press, {\tt https://doi.org/10.1090/mcom/3998}], to develop a new fully-discrete OESV method. The OE technique is non-intrusive, efficient, and straightforward to implement, acting as a simple post-processing filter to effectively suppress spurious oscillations. 
From a DG perspective, we present a comprehensive framework to theoretically analyze the stability and accuracy of both general Runge-Kutta SV (RKSV) schemes and the novel OESV method. For the linear advection equation, we conduct an energy analysis of the fully-discrete RKSV method, identifying an upwind condition crucial for stability. Furthermore, we establish optimal error estimates for the OESV schemes, overcoming nonlinear challenges through error decomposition and treating the OE procedure as additional source terms in the RKSV schemes. 
Extensive numerical experiments validate our theoretical findings and demonstrate the effectiveness and robustness of the proposed OESV method. This work enhances the theoretical understanding and practical application of SV schemes for hyperbolic conservation laws, making the OESV method a promising approach for high-resolution simulations.

\end{abstract}

% REQUIRED

% REQUIRED

\section{Introduction}

Hyperbolic conservation laws are a system of hyperbolic partial differential equations that model conservation principles in continuum physics. 
Due to the complexity of obtaining analytical solutions, numerical simulation has become a crucial tool for studying nonlinear hyperbolic conservation laws.  
 It is well-known that solutions to these nonlinear equations, even with smooth initial and boundary conditions, can develop discontinuities in a finite time. 
Without appropriate treatments, high-order numerical schemes can produce spurious oscillations near these discontinuities, leading to numerical instability and potentially causing the computational codes to blow up. This presents significant challenges in the simulation of hyperbolic conservation laws. Over the past decades, various effective numerical methods have been developed, including the discontinuous Galerkin (DG) method and the spectral volume (SV) method.

The DG method is a class of finite element methods first introduced in  \cite{Reed1973TriangularMM}. Unlike traditional finite element methods, the DG method seeks numerical approximations in discontinuous piecewise polynomial spaces. Due to the pioneering works of Cockburn and Shu \cite{Cockburn1988TheRL,Cockburn1989TVBRL,COCKBURN198990,Cockburn1990TheRL,COCKBURN1998199}, the Runge--Kutta DG (RKDG) method, which couples the DG method with RK time discretization, has become one of the most popular approaches for solving hyperbolic conservation laws. 
The mathematical theory of both the semi-discrete and fully-discrete DG method has been extensively investigated. This includes studies on $L^2$-stability \cite{article5,Zhang2010StabilityAA,Xu2019TheLS,Sun2018StrongSO}, error estimates \cite{article1,article8-1,Xu2020ErrorEO,article2}, and superconvergence analysis \cite{article9,article10,Xu2020SuperconvergenceAO,article3}. 
To mitigate spurious oscillations in the DG method, several techniques have been developed. One effective approach is the application of limiters \cite{article4,article5,Zhong2013ASW}. Another approach involves adding artificial diffusion terms to the weak formulations \cite{ZINGAN2013479,article6,CiCP-27-1309,article7}. 
Recently, Lu, Liu, and Shu \cite{Liu2021AnOF,Liu2022AnEO} introduced the oscillation-free DG (OFDG) method, which controls oscillations by incorporating artificial damping terms into the semi-discrete DG formulation. Motivated by the damping technique of OFDG method, Peng, Sun, and Wu developed the oscillation-eliminating DG (OEDG) method \cite{Peng2023OEDGOD}, which integrates a non-intrusive scale-invariant  oscillation-eliminating  (OE) procedure after each RK stage to eliminate spurious oscillations.

Similar to the DG method, the SV method also employs discontinuous piecewise polynomials as its solution space. Because of using piecewise constant functions as the test functions, the SV method can be considered as a Petrov--Galerkin method. The discretization of the SV method ensures that the conservation laws are satisfied at the sub-element level, allowing it to potentially  achieve higher resolution near discontinuities compared to the DG method \cite{Sun2004EvaluationOD}. Initially proposed by Wang and Liu \cite{WANG2002210,WANG2002665,Wang2004SpectralV}, the SV method has been successfully applied to various physical systems, such as the Euler equations \cite{Wang2004SpectralV,LIU2017120}, the shallow water equations \cite{ASpectralFiniteVolumeMethodfortheShallowWaterEquations,article8}, the Navier--Stokes equations \cite{Sun2006SpectralV,Haga2009RANSSU}, and the Maxwell equations \cite{LIU2006454}. 
Compared to the DG method, far fewer studies have focused on the mathematical theory of the SV method. Most of the theoretical analyses available pertain only to the semi-discrete SV method. For instance, in \cite{ZHANG2005581}, Zhang and Shu discussed the stability of first, second, and third-order semi-discrete SV schemes using Fourier-type analysis. Van den Abeele et al.~explored the wave propagation properties of the semi-discrete SV method for hyperbolic equations in \cite{VANDENABEELE2007616,VANDENABEELE20071007} and later applied the matrix method to investigate the stability of second and third-order semi-discrete SV schemes on 3D tetrahedral grids in \cite{VANDENABEELE2009257}. 
Recently, Cao and Zou analyzed semi-discrete SV schemes based on two specific types of subdivision points for 1D linear hyperbolic conservation laws \cite{Cao2021AnalysisOS}. They developed a novel from-trial-to-test-space mapping and derived a Galerkin form of the SV method, through which they established stability, optimal convergence rates, and some superconvergence properties. They also discovered that, for linear advection equations, the DG method is equivalent to the SV method with a proper choice of subdivision points. Based  on  \cite{Cao2021AnalysisOS}, Zhang, Cao, and Pan \cite{article11} proposed a semi-discrete oscillation-free SV method for hyperbolic conservation laws. 
Further extensions of Cao and Zou's analysis to more complex equations can be found in \cite{Cao2023UnifiedAO,An2023AnyOS}. Notably, Lu, Jiang, Shu, and Zhang \cite{LuAnalysisOA} extended the analysis of the SV method to a broader class of subdivision points. 

Exploring the stability and convergence analysis of the fully-discrete SV method is important yet nontrivial. This remains largely unexplored especially for {\em nonlinear} SV schemes with an automatic oscillation control mechanism.  
To the best of our knowledge, the only theoretical analysis for the fully-discrete (linear) SV method was recently given by Wei and Zou in \cite{Wei2023AnalysisOT}. 
In this work, the authors restricted the subdivision points to either the right Radau points or the Gauss quadrature points and focused on the SV schemes coupled with only the forward Euler and the second-order strong-stability-preserving Runge-Kutta (RK) time discretizations. Wei and Zou \cite{Wei2023AnalysisOT} derived stability and optimal error estimates by examining the temporal difference terms and the error equations of the fully-discrete SV schemes. 
However, the analysis in \cite{Wei2023AnalysisOT} relies on the specific formulation of the temporal difference terms and the error equations corresponding to the aforementioned first- and second-order RK time discretizations. As a result, this technique may not be directly applicable to the general fully-discrete SV method coupled with arbitrary (higher-order) RK time discretization.

	Motivated by the Galerkin form of the SV method in \cite{Cao2021AnalysisOS}, 
	this paper establishes a closer connection between the SV method and the DG method for hyperbolic conservation laws. Building on this connection, we present a comprehensive  framework to theoretically analyze the stability and error estimates of the general RKSV schemes from a DG perspective, and propose a novel fully-discrete oscillation-eliminating SV (OESV) method. The key contributions and findings of this work are summarized as follows:
\begin{itemize}[leftmargin=*]
    \item Based on the from-trial-to-test-space operator $M^*$ proposed in \cite{Cao2021AnalysisOS}, we carefully investigate the biliniear form $\ip{\cdot}{M^*\cdot}$, which is crucial in the Galerkin form of the SV method. We discover that this bilinear form becomes an inner product $\ip{\cdot}{\cdot}_*$ on the discontinuous finite element space if and only if \cref{assump:subdivision} holds. \cref{assump:subdivision} generalizes the restriction on the subdivision points considered in the existing works \cite{Cao2021AnalysisOS,LuAnalysisOA}. Under \cref{assump:subdivision}, we derive an important DG representation of the SV method (\cref{thm:dgr-1DsemiSV}). This representation reveals a closer connection between the SV method and the DG method. Moreover, our DG representation of the SV method suggests that, by replacing the standard $L^2$-inner product with $\ip{\cdot}{\cdot}_*$, the stability and error estimates of the fully-discrete RKSV method can be analyzed within the framework of corresponding DG results.

    \item In view of the DG representation of the SV method, we propose  fully-discrete OESV schemes by extending the OE procedure designed for the RKDG method \cite{Peng2023OEDGOD} to the RKSV scheme. 
    The OESV method alternates between evolving the conventional semi-discrete SV scheme and a damping equation, whose solution operator explicitly defines the OE procedure for suppressing spurious oscillations without requiring characteristic decomposition. 
    The OE procedure is  
    scale-invariant, evolution-invariant, and effective for problems across different scales and
    wave speeds. 
    Furthermore, it is non-intrusive, efficient, and easy to implement, acting as a simple post-processing filter based on only jump information. 
    Extensive benchmark examples are tested, demonstrating the accuracy, effectiveness, and robustness of our OESV schemes.
    \item For the linear advection equation, we carry out an energy analysis for the fully-discrete RKSV method by introducing the temporal difference operator. We show that, to ensure the RKSV schemes are stable in the sense of \eqref{eq:stability-RKSV}, an upwind condition should be satisfied (\cref{thm:stability-RKSV}). Under the upwind condition, we observe that the stability results of the RKDG method can be directly extended to the RKSV schemes (\cref{rmk:DGSVsameresult}). We also prove that the upwind condition restricts the choice of the subdivision points to the zeros of a specific class of polynomials (\cref{thm:pik}).
     
    \item We establish the optimal error estimates for the fully-discrete OESV schemes (\cref{thm:err-OESV}). This proof is nontrivial since the OESV schemes are essentially nonlinear. The key idea of our analysis is to decompose the error through a piecewise polynomial interpolation $P^*$ (\cref{def:Pstar}) and skillfully treat the OE procedure as  extra source terms in the RKSV schemes (\cref{prop:OESV-xi}).
\end{itemize}

The paper is organized as follows. In section 2, we give a brief introduction to the SV and OESV method for general hyperbolic systems of conservation laws. In section 3, we specifically consider the SV method for 1D hyperbolic conservation laws and derive the DG representation. Stability analysis and the optimal error estimates of the fully-discrete RKSV and OESV schemes are established in section 4. Section 5 presents benchmark numerical examples of the OESV scheme. Concluding remarks are provided in section 6. Throughout this paper, $\Omega$ denotes the spatial domain and $I_i=[ x_{i-\frac{1}{2}}, x_{i+\frac{1}{2}} ]$ is a 1D cell. Inner product $\ip{\cdot}{\cdot}$ and $\ipi{\cdot}{\cdot}$ represents the standard $L^2$-inner product over $\Omega$ and $I_i$, respectively. $\nm{\cdot}$ represents the $L^2$-norm on $\Omega$. The following notations relating to jumps at each $x_{i+\hf}$ will also be used:
\begin{equation}\label{jump}
    \jump{v}_{i+\hf} = v\left( x_{i+\frac{1}{2}}^+ \right) - v\left( x_{i+\frac{1}{2}}^- \right),~~ \ipjump{v}{\omega} = \sum_i \jump{v}_{i+\hf}\jump{\omega}_{i+\hf},~~ \jump{v} = \sqrt{\ipjump{v}{v}}.
\end{equation}

\section{SV and OESV methods} 
In this section, we introduce the SV method and propose the novel OESV method for the general hyperbolic system of conservation laws:
\begin{equation}\label{eq:HCL}
    \begin{cases}
        {\bf u}_t + \nabla \cdot {\bf f}({\bf u}) = {\bf 0},\ \ & ({\bf x}, t) \in \Omega \times \left( 0, T \right] \\
        {\bf u}({\bf x}, 0) = {\bf u}_0({\bf x}), \ \ & {\bf x} \in \Omega,
    \end{cases}
\end{equation}
where ${\bf u} \in \mathbb R^N$, ${\bf f}=({\bf f}_1,\dots,{\bf f}_d)$, ${\bf x}=(x_1,\dots,x_d)$,  and $\Omega$ is a bounded domain in $\mathbb{R}^d$.

\subsection{Semi-discrete SV formulation}
Let $\mathcal{T}_h$ be a partition of the spatial domain $\Omega$. The discontinuous finite element space over $\mathcal{T}_h$ is defined as
\begin{equation}\label{eq:DGspace}
	{\mathbb V}^k := \left \{ v \in L^2(\Omega):~ v |_K \in [\mathbb P^k(K)]^N~ or~ [\mathbb Q^k(K)]^N\quad \forall K \in \mathcal{T}_h \right \}.
\end{equation}
Here, $\mathbb P^k(K)$ denotes the polynomial space with degree less than or equal to $k$, and $\mathbb Q^k(K)$ denotes the bi-$k$ tensor product polynomial space on element $K$. Assume that each element $K$ has a subdivision $\mathcal{T}_{K}^{*}$ satisfying
\begin{equation*}
    \bigcup_{K^* \in \mathcal{T}_{K}^{*}} K^* = K,\quad \left| \mathcal{T}_{K}^{*} \right| = \begin{cases}
\dim \mathbb P^k(K),~~ &if\ \mathbb P^k \ elements, \\
\displaystyle
\dim \mathbb Q^k(K),~~ &if\ \mathbb Q^k \ elements,
\end{cases}
\end{equation*}
where $K^* \in \mathcal{T}_{K}^{*}$ is called a {\em control volume (CV)}. The semi-discrete SV scheme for hyperbolic conservation laws \eqref{eq:HCL} seeks the numerical solution $\mathbf{u}_h(\cdot,t) \in {\mathbb V}^k$ based on the information in each CV:
\begin{equation}\label{eq:semiSV}
    \int_{K^{*}} ({\bf u}_h)_t {\rm d} {\bm x} = -\sum_{e \in \partial K^{*}} \int_{e} \widehat {\bf f}_e ({\bf u}_h) \cdot {\bf n}_e {\rm d} S \quad \forall K^{*} \in \mathcal{T}_{K}^{*},\ \forall K \in \mathcal{T}_h.
\end{equation}
In \eqref{eq:semiSV}, ${\bf n}_e=(n_e^{(1)},\dots,n_e^{(d)})$ denotes the outward unit vector at $e$ with respect to $K^{*}$, and $\widehat {\bf f}_e ({\bf u}_h)$ is a suitable numerical flux on the interface $e \in \partial K^{*}$. For $e$ in the interior of $K$, since ${\bf u}_h$ is smooth across $e$, we define $\widehat {\bf f}_e ({\bf u}_h)$ to be ${\bf f} ({\bf u}_h)$.

\subsection{Runge--Kutta SV method}
The semi-discrete scheme \eqref{eq:semiSV} can be treated as an ODE system ${\frac{\mathrm{d}}{\mathrm{d} t}} {\bf u}_h = L_{\bf f}({\bf u}_h)$, which can then be further discretized in time $t$ using, for example, some high-order accurate RK or multi-step methods. In this paper, we focus on the RKSV scheme, which is the fully-discrete SV scheme coupled with an $r$th-order $s$-stage RK method:
\begin{equation}\label{eq:RKSV}
\begin{aligned}
{\bf u}_h^{n,0} &= {\bf u}_h^{n},  \\
{\bf u}_h^{n,\ell+1} &= \sum_{0\leq \kappa \leq \ell} \left(c_{\ell \kappa}{\bf u}_h^{n,\kappa} + \tau d_{\ell \kappa} L_{\bf f}({\bf u}_h^{n,\kappa}) \right),\quad \ell =0,1,\dots,s-1, \\
{\bf u}_h^{n+1} &= {\bf u}_h^{n,s}.
\end{aligned} 
\end{equation}
Here, $\tau$ is the time step-size, ${\bf u}_h^{n}$ is the numerical solution at the $n$th time step, and the constants $\left\{ c_{\ell \kappa} \right\}$,  $\left\{ d_{\ell \kappa} \right\}$ are determined by the RK method with $\sum_{0\leq \kappa \leq \ell}c_{\ell \kappa} = 1$.

\subsection{OESV method}\label{sec:OESV}
The above RKSV scheme \eqref{eq:RKSV} works well in smooth regions but typically generates spurious oscillations near discontinuities.  
To address this issue, we propose the fully-discrete OESV method, which incorporates  
an OE procedure after each RK stage of the RKSV scheme: 
\begin{align}
{\bf u}_h^{n,0} &= {\bf u}_h^n, 
\\
\tilde{\bf u}_h^{n,\ell+1} &= \sum_{0\leq \kappa \leq \ell} \left(c_{\ell \kappa} {\bf u}_h^{n,\kappa} + \tau d_{\ell \kappa} L_{\bf f}( {\bf u}_h^{n,\kappa}) \right),\label{eq:rk}\\
{\bf u}_h^{n,\ell+1} &= \mathcal{F}_\tau \tilde{\bf u}_h^{n,\ell+1},\qquad \ell =0,1,\dots,s-1, \label{eq:OEstep}
\\
{\bf u}_h^{n+1} &= {\bf u}_h^{n,s},
\end{align}
where \eqref{eq:OEstep} represents the OE procedure  \cite{Peng2023OEDGOD}, and $\mathcal{F}_\tau \tilde{\bf u}_h^{n,\ell+1}({\bf x}) = {\bf u}({\bf x},\tau)$ with ${\bf u}({\bf x},\hat t) \in {\mathbb V}^k$ ($0 \le \hat t \le \tau$) defined as the solution of the following damping equations:
\begin{equation}\label{eq:filter}
\left\{  
\begin{aligned}
&{\frac{\mathrm{d}}{\mathrm{d} \hat t}  \int_K {\bf u}\cdot {\bf v} {\rm d}{\bm{x}}} + \sum_{m = 0}^{k}
\delta_K^{m} (\tilde{\bf u}_h^{n,\ell+1})  {\int_K( {{\bf u} - \mathbf{P}^{m-1}{\bf u}})\cdot{\bf v}{\rm d}{\bm x}} = 0, ~~\forall K, \forall {\bf v} \in {\mathbb V}^k,\\
&{\bf u}(x,0) =\; \tilde{\bf u}_h^{n,\ell+1}(x).
\end{aligned} \right.
\end{equation}
In \eqref{eq:filter}, $\hat t$ is a pseudo-time different from $t$, and $\mathbf{P}^{m}$ is the standard $L^2$-projection into ${\mathbb V}^m$ for $m\geq 0$, with $\mathbf{P}^{-1}$ set as $\mathbf{P}^{0}$. The damping coefficient $\delta_K^{m} ({\bf u}_h)$ is defined as
\begin{equation*}\label{eq:delta}
\delta_K^{m} ({\bf u}_h) = \sum_{e \in \partial K} \beta_e \frac{ \sigma_{e,K}^m({\bf u}_h) }{h_{e,K}},\qquad \sigma_{e,K}^m({\bf u}_h) = \max_{1\le i\le N} \sigma_{e,K}^m(u_h^{(i)}),
\end{equation*}
where $h_{e,K} = \sup_{\bm x \in K} {\rm dist}({\bm x},e)$,  
$\beta_e$ is the spectral radius of $\sum_{i=1}^d n_e^{(i)} \frac{\partial {\bf f}_i}{\partial {\bf u}} \big|_{{\bf u}=\overline{\bf u}_K}$ with $\overline{\bf u}_K$ being the  average of  ${\bf u}_h({\bf x})$ over $K$, $u_h^{(i)}$ denotes the $i$th component of ${\bf u}_h$, and 
\begin{equation}\label{eq:sig}
\sigma_{e,K}^m(u_h^{(i)}) = 
\begin{cases} 
0, ~~ & {\rm if}~  u_h^{(i)} \equiv \mathrm{avg}_{\Omega}(u_h^{(i)}), \\
\displaystyle
\frac{ (2m+1)h_{e,K}^m}{2(2k-1)m!}  \sum_{|{\bm \alpha }|=m} \frac{ \frac{1}{|e|} \int_{e}  \big|[[\partial^{\bm  \alpha} u_h^{(i)}]]_e \big| {\rm d}S  } { \|  u_h^{(i)} - \mathrm{avg}_{\Omega}(u_h^{(i)})  \|_{L^\infty(\Omega)}  }, ~~ & {\rm otherwise}.
\end{cases}
\end{equation}
Here, $\mathrm{avg}_{\Omega}(u_h^{(i)})$ denotes the {\bf global} average of $u_h^{(i)}({\bf x})$ over the {\bf entire} computational domain $\Omega$, and   
$|e|$ is the length or surface area of $e \in \partial K$. The multi-index vector ${\bm \alpha} = ( \alpha_1,\dots, \alpha_d)$, $\partial^{\bm  \alpha} u_h^{(i)} = \frac{ \partial^{ \alpha_1+\cdots+\alpha_d } }{\partial x_1^{ \alpha_1} \cdots \partial x_d^{ \alpha_d}} u_h^{(i)}$, and $|{\bm \alpha }|$ is defined as $\sum_{i=1}^{d}\alpha_i$ for $\mathbb P^k$ elements and 
$\max_{1\le i \le d}\alpha_i$ for $\mathbb Q^k$ elements. 
%\begin{equation*}
%|{\bm \alpha }|=\begin{cases}
%\sum_{i=1}^{d}\alpha_i,~~ &\mbox{for}~~ \mathbb P^k ~\mbox{elements}, \\
%\displaystyle
%\max_{i=1,\dots,d}\alpha_i,~~ &\mbox{for}~~ \mathbb Q^k ~\mbox{elements},
%\end{cases}
%\end{equation*}
In \eqref{eq:sig}, $|[[\partial^{\bm  \alpha} u_h^{(i)}]]_e|$ denotes the absolute value of the jump of $\partial^{\bm  \alpha} u_h^{(i)}$ across interface $e$. In multidimensional cases, the integral $\int_{e}  \big|[[\partial^{\bm  \alpha} u_h^{(i)}]]_e \big| {\rm d}S$ should be approximated using a suitable quadrature, and we adopt the simple trapezoidal rule for $d=2$. In our computations, we evaluate $ \|  u_h^{(i)} - \mathrm{avg}_{\Omega}(u_h^{(i)})  \|_{L^\infty(\Omega)}$ on the Gauss quadrature nodes.

The OE procedure \eqref{eq:OEstep} is efficient and easy to implement, since the solution operator $\mathcal{F}_\tau$ to the 
damping equations \eqref{eq:filter} can be explicitly formulated as 
\begin{equation*}
	\cF_\tau \tilde{\bf u}_h^{n,\ell+1}  =  {\bf u}_K^{({\bm 0})} (0) \phi_K^{({\bm  0})}({\bm x}) + \sum_{j=1}^k {\rm e}^{-\tau \sum_{m=0}^j \delta_K^m( \tilde{\bf u}_h^{n,\ell+1} )  } \sum_{ |{\bm  \alpha}|=j } {\bf u}_K^{({\bm  \alpha})} (0) \phi_K^{({\bm  \alpha})}({\bm x}), 
\end{equation*}
where 
${
	{\bf u}_K^{({\bm  \alpha})} (0) = \int_K \tilde{\bf u}_h^{n,\ell+1} \phi_K^{({\bm  \alpha})} {\rm d}\bm{x}} / \| \phi_K^{({\bm  \alpha})} \|_{L^2(K)}^2
$, and $\{\phi_{K}^{(\bm  \alpha)}\}_{|\bm  \alpha|\le k}$ is an orthogonal basis of $\mathbb P^k(K)$. 
The OE procedure is  
scale-invariant, evolution-invariant, and thus effective for problems across various scales and
wave speeds \cite{Peng2023OEDGOD}.

\section{Understanding SV from a DG perspective}
In this section, we present a novel understanding of the SV method from a DG viewpoint. 
For the sake of convenience, we focus on the 1D scalar hyperbolic conservation law: 
\begin{equation}\label{eq:1DHCL}
u_t + f(u)_x = 0,\quad x \in \Omega = [a,b],~ t \in \left( 0, T \right]. 
\end{equation}
The extension to hyperbolic systems is straightforward. We will derive an important DG representation of the SV method. This representation will lead to a comprehensive framework for analyzing the SV method using the existing DG techniques and results. 

\subsection{Preliminary} 
This subsection discusses the Galerkin form of the SV method proposed by Cao and Zou in \cite{Cao2021AnalysisOS} and its extension to general subdivision points.  
Let $\Omega=\cup_{i}I_{i}$ be a partition of the 1D bounded spatial domain $\Omega$, consisting of a finite number of cells $I_{i}=[ x_{i-\frac{1}{2}}, x_{i+\frac{1}{2}} ]$. Assume that the partition is quasi-uniform, meaning that there exists a constant $C>0$ such that $h\leq C  h_i$ for all $i$. Here, $h_i = x_{i+\frac{1}{2}} - x_{i-\frac{1}{2}}$ denotes the mesh size of $I_i$ and $h = \max_i h_i$. The discontinuous finite element space is 
$
    \mathbb{V}^{k}=\left\{ v\in L^{2}(\Omega): v|_{I_{i}}\in \mathbb{P}^{k}(I_{i})~~ \forall i \right\}.
$ 
In this context, the conventional semi-discrete DG scheme seeks $u_h(\cdot,t) \in {\mathbb V}^k$ such that
\begin{equation}\label{eq:1DsemiDG}
    \ip{(u_h)_t}{\omega} = \sum_i \left( \ipi{f\left(u_h\right)}{\omega_x}  + \omega_{i-\frac{1}{2}}^{+}\hat{f}_{i-\frac{1}{2}}-\omega_{i+\frac{1}{2}}^{-}\hat{f}_{i+\frac{1}{2}} \right)\quad \forall \omega \in {\mathbb V}^k,
\end{equation}
where $\hat{f}_{i+\frac{1}{2}}$ is the numerical flux at $x_{i+\hf}$. Assume that each $I_{i}$ has a subdivision $I_{i} = \cup_{j=0}^{k}I_{i,j}$, where the control volume $I_{i,j}=\left[ x_{i,j}, x_{i,j+1} \right]$ with $x_{i,0}=x_{i-\frac{1}{2}}$ and $x_{i,k+1}=x_{i+\frac{1}{2}}$. If we introduce the piecewise constant space over the control volumes:
\begin{equation*}
    \mathbb{V}^{k,*}=\left\{ v^*\in L^{2}(\Omega) \ : \ v^*|_{I_{i,j}}\in \mathbb{P}^{0}(I_{i,j})\quad \forall i,j \right\},
\end{equation*}
then, by \eqref{eq:semiSV} and \cite{Cao2021AnalysisOS}, the semi-discrete SV scheme for \eqref{eq:1DHCL} can be expressed as
\begin{equation}\label{eq:1DsemiSV}
\ip{(u_{h})_t}{\omega^{*}}=\sum_{i}\sum_{j=0}^{k}\omega_{i,j}^{*}(\hat{f}_{i,j}-\hat{f}_{i,j+1})\quad \forall \omega^{*}\in \mathbb{V}^{k,*},
\end{equation}
where $\hat{f}_{i,j}$ denotes the numerical flux at $x_{i,j}$, and $\omega_{i,j}^*$ is the constant value of $\omega^*$ on $I_{i,j}$.

To obtain a Galerkin form of \eqref{eq:1DsemiSV}, Cao and Zou \cite{Cao2021AnalysisOS} introduce the following quadrature rules within each $I_i$: 
\begin{equation}\label{eq:quadrature}
    Q_i^k(v) = \sum_{j=0}^{k+1}A_{i,j}v(x_{i,j}),~ R_i^k(v)=\int_{I_i}v {\rm d} x - Q_i^k(v).
\end{equation}
Here, $A_{i,j}$ is the quadrature weight at $x_{i,j}$, and $R_i^k$ denotes the error between the numerical quadrature and the exact integral. By the theory of quadrature rules  \cite{Brass2011QuadratureTT}, there always exists a quadrature rule $Q_i^k$ such that $R_i^k(v)=0$ for any $v \in \mathbb{P}^{k-1}(I_{i})$. Hence, we assume that $Q_i^k$ is exact for all polynomials of degree $\le k-1$ in the following.  
Based on \eqref{eq:quadrature}, Cao and Zou \cite{Cao2021AnalysisOS} construct a from-trial-to-test-space mapping $M^*$ as follows. 
\begin{definition}[\cite{Cao2021AnalysisOS}]\label{def:Mstar} The from-trial-to-test-space mapping $M^*$ is defined as
    \begin{equation}\label{eq:Mstar}
        \left( M^*\omega \right) |_{I_{i}} = M_i^*\left( \omega |_{I_{i}} \right)\quad \forall i,
    \end{equation}
    where $M_i^*$ is defined recursively by 
    \begin{equation}\label{eq:Mistar}
        \begin{cases}
            & \left( M_i^{*}v \right)|_{I_{i,0}}=v_{i-\frac{1}{2}}^{+}+A_{i,0}v_x(x_{i-\frac{1}{2}}),\\
            & \left( M_i^{*}v \right)|_{I_{i,j}}-\left( M_i^{*}v \right)|_{I_{i,j-1}}=A_{i,j}v_x(x_{i,j}),\ j=1,...,k.
        \end{cases}
\end{equation}
\end{definition}

The 
following important property of $M^*$ was proven in \cite{Cao2021AnalysisOS}. 

\begin{lemma}[{\cite[Theorem 3.1]{Cao2021AnalysisOS}}]\label{lemma:Mstar-difference}
	For any $v \in L^2(\Omega)$, if there exists $V \in L^2(\Omega)$ such that $V_x = v$, then
	\begin{equation}\label{eq:Mstar-difference}
		\ip{v}{M^*\omega} = \ip{v}{\omega} + \sum_i R_i^k(V\omega_x)\quad \forall \omega \in H^1(\Omega).
	\end{equation}
\end{lemma}

We observe that $M^*$ is a surjective linear operator, regardless of the choice of subdivision points $\{ x_{i,j} \}$ and the quadrature rules $\{ Q_i^k \}$. 
\begin{proposition}\label{prop:Mstar-surjective}
    The operator $M^*$ defined by \eqref{eq:Mstar} is surjective.
\end{proposition}

The proof \Cref{prop:Mstar-surjective} is presented in \Cref{app:1}. It is worth noting that, 
if the subdivision points $\{ x_{i,j} \}$ are taken as the Gauss or right Radau points, 
then \cite{Wei2023AnalysisOT} proved that $M^*$ is bijective.

Regardless of the choice of subdivision points $\{ x_{i,j} \}$ and the quadrature rules $\{ Q_i^k \}$, 
according to \cref{prop:Mstar-surjective}, the semi-discrete SV scheme \eqref{eq:1DsemiSV} is equivalent to the following ``Galerkin'' form
\begin{equation}\label{eq:1DsemiSV-G}
\ip{(u_{h})_t}{M^* \omega}=\sum_{i}\sum_{j=0}^{k}\left( M^* \omega \right)_{i,j}(\hat{f}_{i,j}-\hat{f}_{i,j+1})\quad \forall \omega\in \mathbb{V}^{k},
\end{equation}
where $\left( M^* \omega \right)_{i,j} = \left( M^* \omega \right) |_{I_{i,j}} $ denotes the constant value of $M^* \omega$ on $I_{i,j}$. The form \eqref{eq:1DsemiSV-G} was first proposed by Cao and Zou \cite{Cao2021AnalysisOS} for two special sets of subdivision points (the Gauss and right Radau points).  

\subsection{DG representation of semi-discrete SV method}

From \eqref{eq:quadrature}, \eqref{eq:Mstar}, and \eqref{eq:1DsemiSV-G}, one can see that the choice of subdivision points  $\{ x_{i,j} \}$ determines the properties of the corresponding SV method. In the following, we consider the SV method with the subdivision points satisfying the following assumption.
\begin{assumption}\label{assump:subdivision}
    There exists a quadrature $Q_i^k$ on the subdivision points $\{ x_{i,j} \}$ such that 
    \begin{enumerate}
        \item $R_i^k(v) = 0$ for all $v \in \mathbb{P}^{2k-1}(I_i)$;
        \item $\frac{h_i}{2k-1} - Q_i^k(L_{i,k+1}L_{i,k-1})>0$, where $L_{i,\ell}$ denotes the shifted Legendre polynomial of degree $\ell$ on $I_i$ satisfying $\ipi{L_{i,\ell}}{L_{i,\ell}}=\frac{h_i}{2\ell+1}$. 
    \end{enumerate}
\end{assumption}

An interesting observation is that the bilinear form $\ip{\cdot}{M^* \cdot}$ is actually an inner product when \cref{assump:subdivision} holds. Specifically, we have the following theorem. 
\begin{theorem}\label{thm:subdivision}
    $\ip{\cdot}{M^* \cdot}$ is an inner product on $\mathbb{V}^{k}$ $\Leftrightarrow$ \cref{assump:subdivision} holds.
\end{theorem}
\begin{proof}
    We first show that 
    \begin{equation}\label{eq:defip1}
        \ipi{\omega}{M_i^* v} = \ipi{v}{M_i^*\omega}\quad \forall v,\omega \in \mathbb{P}^{k}(I_i) ~ \Leftrightarrow ~ R_i^k(v) = 0\quad \forall v \in \mathbb{P}^{2k-1}(I_i).
    \end{equation}
    $(\Leftarrow)$: Without the loss of generality, we assume that $v=\sum_{l=0}^{k}v_l x^l$ and $\omega = \sum_{l=0}^{k}\omega_l x^l$. Then by \cref{lemma:Mstar-difference},
    \begin{equation}
        \ipi{\omega}{M_i^* v} = \ipi{\omega}{v} + \frac{k}{k+1}\omega_k v_k R_i^k(x^{2k}) = \ipi{v}{M_i^*\omega}.
    \end{equation}
    $(\Rightarrow)$: For $0\leq s \leq k-1$, consider $p_s = x^s + x^k$ and $q_s = 2x^s + x^k$. By \cref{lemma:Mstar-difference},
    \begin{equation*}
        \left( \frac{k}{s+1} - \frac{s}{k+1} \right) R_i^k(x^{k+s}) = \ipi{q_s}{M_i^* p_s} - \ipi{p_s}{M_i^* q_s} = 0.
    \end{equation*}
    Hence, $R_i^k(x^{k+s})=0$ for $s=0,\dots,k-1$, which indicates $R_i^k(v) = 0$ for all $v \in \mathbb{P}^{2k-1}(I_i)$.

    Next, assume that $R_i^k(v) = 0$ for all $v \in \mathbb{P}^{2k-1}(I_i)$. Then, by  \cref{lemma:Mstar-difference},
    \begin{equation*}
        \ipi{v}{M_i^*v} >0\quad \forall v \in \mathbb{P}^{k}(I_i)\setminus\{0\}  ~ \Leftrightarrow ~ \ipi{L_{i,k}}{M_i^* L_{i,k}}>0.
    \end{equation*}
    Note that the Legendre polynomials in $I_i$ satisfy $L_{i,\ell}(x_{i-\hf}) = (-1)^\ell$, $L_{i,\ell}(x_{i+\hf})=1$, $\ipi{L_{i,\ell}}{L_{i,\ell}}=\frac{h_i}{2\ell+1}$, $(2k+1)L_{i,k} = \frac{h_i}{2}\left( L_{i,k+1}^{'}- L_{i,k-1}^{'}\right)$, and
    \begin{equation*}
        (2k+1)L_{i,k} = \frac{h_i}{2}\left( L_{i,k+1}^{'}- L_{i,k-1}^{'}\right),\quad \frac{h_i}{2}L_{i,k}^{'} = (2k-1)L_{i,k-1} + (2k-5)L_{i,k-3} + \cdots.
    \end{equation*}
    It follows that
    \begin{equation}\label{eq:normLik}
        \begin{aligned}
            \ipi{L_{i,k}}{M_i^* L_{i,k}} & = \frac{h_i}{2(2k+1)}\left( Q_i^k(L_{i,k-1}L_{i,k}^{'}) - Q_i^k(L_{i,k+1}L_{i,k}^{'}) \right) \\
            & = \frac{2k-1}{2k+1}\left( \frac{h_i}{2k-1} - Q_i^k(L_{i,k+1}L_{i,k-1}) \right).
        \end{aligned}
    \end{equation}
    Hence, $\ipi{\cdot}{M_i^* \cdot}$ is an inner product if and only if $\frac{h_i}{2k-1} - Q_i^k(L_{i,k+1}L_{i,k-1})>0$.
\end{proof}

\begin{remark}
	Although \eqref{eq:Mstar-difference} is valid for $v,\omega \not \in \mathbb{V}^k$, 
	it is important to note that 
	 $\ip{\cdot}{M^* \cdot}$ forms an inner product \textbf{only on} the piecewise polynomial space $\mathbb{V}^k$. % under \cref{assump:subdivision}. 
\end{remark}

\begin{remark}
    \cite{Cao2021AnalysisOS,Cao2023UnifiedAO,An2023AnyOS,Wei2023AnalysisOT} have studied the SV methods with the Gauss and right Radau nodes as the subdivision points. 
     For these special subdivision points, \cite{Wei2023AnalysisOT} also noticed that $\ip{\cdot}{M^* \cdot}$ is an inner product on $\mathbb{V}^{k}$.  
    \cref{assump:subdivision} is, in fact, a generalization of the restrictions on the subdivision points discussed in  \cite{Cao2021AnalysisOS,Cao2023UnifiedAO,An2023AnyOS,Wei2023AnalysisOT}. 
\end{remark}

Under \cref{assump:subdivision}, we can verify that $M^*$ is bounded. 
\begin{proposition}\label{prop:Mstar}
    $\nm{M^*} := \sup_{\omega \not \equiv 0} \frac{\nm{M^*\omega}}{\nm{\omega}} \leq C$, where $C$ is a positive constant independent of the mesh $\{ I_i \}$.
\end{proposition}

The proof of \Cref{prop:Mstar} is given in \Cref{app:2}.

\begin{definition}\label{def:ipstar}
    Define the inner product $\ip{\cdot}{\cdot}_*: \mathbb{V}^k \times \mathbb{V}^k \to \mathbb{R}$ by
    \begin{equation*}
        \ip{v}{\omega}_* = \ip{v}{M^* \omega}\quad \forall v,\omega \in \mathbb{V}^k,
    \end{equation*}
    and denote $\nm{\cdot}_*=\sqrt{\ip{\cdot}{\cdot}_*}$. We refer to $\nm{\cdot}_*$ as the energy norm on $\mathbb{V}^k$.
\end{definition}
%we can also prove that:
\begin{proposition}\label{prop:ipstar}
    For any $v \in \mathbb{V}^{k}$, we have 
            \begin{equation}\label{eq:propipstar_1}
    	\ip{v}{\omega}_* = \ip{v}{\omega}\quad \forall \omega \in \mathbb{V}^{k-1}.
    \end{equation}
	In addition, the energy norm $\nm{\cdot}_*$ is equivalent to the $L^2$-norm $\nm{\cdot}$ on $\mathbb{V}^k$, namely, 
        \begin{equation}\label{eq:propipstar_2}
            c\nm{v} \leq \nm{v}_* \leq C\nm{v}\quad \forall v \in \mathbb{V}^{k}
        \end{equation}
        for some positive constants $c$ and $C$ independent of $\{ I_i \}$.
\end{proposition}

The proof of \Cref{prop:ipstar} is presented in \Cref{app:3}.

Based on \eqref{eq:1DsemiSV-G} and \cref{def:ipstar}, we reformulate the right hand side of \eqref{eq:1DsemiSV-G} and derive a DG representation of the SV method as follows. 
\begin{theorem}[DG representation of SV method]\label{thm:dgr-1DsemiSV}
    The semi-discrete SV scheme \eqref{eq:1DsemiSV} is equivalent to the following scheme:
    \begin{equation}\label{eq:dgr-1DsemiSV}
        \ip{(u_{h})_t}{\omega}_* = \sum_i \left( Q_i^k\big( \hat{f} \omega_x \big)+ \omega_{i-\frac{1}{2}}^{+}\hat{f}_{i-\frac{1}{2}}-\omega_{i+\frac{1}{2}}^{-}\hat{f}_{i+\frac{1}{2}} \right)\quad \forall \omega \in \mathbb{V}^{k},
    \end{equation}
    where $\hat{f}$ denotes the numerical fluxes, and $Q_i^k\big( \hat{f} \omega_x \big) = \sum_{j=0}^{k+1}A_{i,j}\hat{f}_{i,j}\omega_x\left( x_{i,j} \right).$
\end{theorem}
\begin{proof}
    We observe that 
    \begin{equation}\label{eq:x_p_0}
        \begin{aligned}
            \sum_{j=0}^{k} & \left( M^* \omega \right)_{i,j} (\hat{f}_{i,j}-\hat{f}_{i,j+1}) = \sum_{j=0}^{k}\left( M^* \omega \right)_{i,j}\hat{f}_{i,j} - \sum_{j=1}^{k+1}\left( M^* \omega \right)_{i,j-1}\hat{f}_{i,j} \\
            = & \sum_{j=1}^{k+1}\hat{f}_{i,j}\left( \left( M^* \omega \right)_{i,j} - \left( M^* \omega \right)_{i,j-1} \right)  + \left( M^* \omega \right)_{i,0}\hat{f}_{i-\hf}-\left( M^* \omega \right)_{i,k+1}\hat{f}_{i+\hf}.
        \end{aligned}
    \end{equation}
    According to \cref{def:Mstar},  we have 
    \begin{align*}
    	& \left( M^* \omega \right)_{i,0} = \omega_{i-\frac{1}{2}}^{+}+ A_{i,0}\omega_x(x_{i,0}), \qquad \left( M^* \omega \right)_{i,j} - \left( M^* \omega \right)_{i,j-1} = A_{i,j}\omega_x(x_{i,j}), 
    	\\
        & \left( M^* \omega \right)_{i,k+1} = \omega_{i-\frac{1}{2}}^{+} + Q_i^k(\omega_x) - A_{i,k+1}\omega_x(x_{i,k+1}) = \omega_{i+\frac{1}{2}}^{-} - A_{i,k+1}\omega_x(x_{i,k+1}),
    \end{align*}
    for $\omega \in \mathbb{V}^k$. Thus we obtain \eqref{eq:dgr-1DsemiSV} by \eqref{eq:1DsemiSV-G} and \eqref{eq:x_p_0}.
\end{proof}

Now, let us understand the novel equivalent form \eqref{eq:dgr-1DsemiSV} by comparing it with the conventional semi-discrete DG scheme \eqref{eq:1DsemiDG}. 
Since $\hat{f}_{i,j}$ is a high-order approximation to the flux $f(u_h)$ at $x_{i,j}$,  $Q_i^k\left( \hat{f} \omega_x \right)$ is essentially a high-order approximation to   $\ipi{f\left(u_h\right)}{\omega_x}$. Note that \eqref{eq:Mstar-difference} and \eqref{eq:propipstar_1} imply that $\ip{\cdot}{\cdot}_*$ is a high-order accurate approximation of the standard $L^2$-inner product on $\mathbb{V}^k$. 
In this sense, \eqref{eq:dgr-1DsemiSV} bridges the SV and DG methods, 
indicating that the SV method can be interpreted as a DG method that replaces $\ip{\cdot}{\cdot}$ with $\ip{\cdot}{\cdot}_*$. More significantly, this DG representation \eqref{eq:dgr-1DsemiSV} of the SV method provides a comprehensive framework for exploring the SV method. This framework allows us to systemically study the stability, optimal error estimates, and  the 
 oscillation elimination techniques of the SV method, based on the existing well-established techniques and results of the DG method.

\subsection{DG representations of fully-discrete RKSV and OESV methods}
To analyze the RKSV and OESV methods later, we give their DG representations 
for the following 1D linear advection equation:
\begin{equation}\label{eq:1Dlinear}
    u_t + \beta u_x = g(x,t),\quad x \in \Omega = [a,b],\quad t \in \left( 0, T \right],
\end{equation}
where the constant coefficient $\beta>0$ without loss of generality. In our analysis, the numerical fluxes will be taken as the upwind flux.

As direct consequences of \Cref{thm:dgr-1DsemiSV}, we have the following DG representations of the RKSV and OESV methods.

\begin{theorem}[DG representation of RKSV method]\label{thm:dgr-RKSV}
   The RKSV scheme for \eqref{eq:1Dlinear} can be equivalently reformulated as: for all $\omega \in \mathbb{V}^k$,
\begin{equation}\label{eq:linearRKSV}
    \ip{u_h^{n,\ell+1}}{\omega}_* = \sum_{0\leq \mykappa \leq \ell} \Big(c_{\ell \mykappa}\ip{u_h^{n,\mykappa}}{\omega}_* + \tau d_{\ell \mykappa}\big(\sv{u_h^{n,\mykappa}}{\omega} {+ \ip{g^{n,\mykappa}}{M^* \omega} } \big)\Big),
\end{equation}
where 
\begin{equation}\label{eq:SVoperator}
	\sv{v}{\omega} = \beta \left( \sum_i Q_i^k(v\omega_x) + \sum_i v_{i+\hf}^-\jump{\omega}_{i+\hf} - \sum_i A_{i,0}(\omega_x)_{i+\hf}^+ \jump{v}_{i+\hf} \right).
\end{equation}
\end{theorem}

\begin{theorem}[DG representation of OESV method]\label{thm:dgr-OESV}
    The OESV scheme for \eqref{eq:1Dlinear} can be equivalently reformulated as: for all $\omega \in \mathbb{V}^k$,
    \begin{subequations}\label{eq:linearOESV}
    \begin{align}
        \ip{\tilde{u}_h^{n,\ell+1}}{\omega}_* & = \sum_{0\leq \mykappa \leq \ell} \left(c_{\ell \mykappa}\ip{u_h^{n,\mykappa}}{\omega}_* + \tau d_{\ell \mykappa}\sv{u_h^{n,\mykappa}}{\omega} {+ \ip{g^{n,\mykappa}}{M^* \omega} } \right). \label{eq:linearRK} \\
        u_h^{n,\ell+1} &= \cF_\tau \tilde{u}_h^{n,\ell+1},\qquad 0\leq \ell \leq s-1. \label{eq:linearOE}
    \end{align}
\end{subequations}
\end{theorem}

In \eqref{eq:linearRKSV} and \eqref{eq:linearOESV}, constants $\left\{ c_{\ell \kappa} \right\}$,  $\left\{ d_{\ell \kappa} \right\}$ are determined by the RK method with $\sum_\kappa c_{\ell \mykappa}=1$ for all $\ell$, as in \eqref{eq:RKSV} and \eqref{eq:rk}. If $g \equiv 0$, the source terms $g^{n,\mykappa} \equiv 0$ for all $\kappa$. Since $g^{n,\mykappa}$ may not belong to $\mathbb{V}^k$, we here use the notation $\ip{g^{n,\mykappa}}{M^* \omega}$ instead of $\ip{g^{n,\mykappa}}{ \omega}_*$, as  $\ip{\cdot}{\cdot}_*$ is an inner product only on $\mathbb{V}^k$.

The quadrature $Q_i^k(\cdot)$ in the SV method might involve the downwind information at each $x_{i-\hf}$, while our numerical fluxes are chosen as the upwind flux. 
If the following condition (referred to as the \textbf{upwind condition}) is satisfied: 
\begin{equation}\label{eq:UWcon}
     \exists ~ Q_i^k ~ \mbox{ such ~ that} ~ A_{i,0}=0\quad \forall i,
\end{equation}
then $Q_i^k(\cdot)$ only requires the upwind information at each $x_{i-\hf}$. When this upwind condition \eqref{eq:UWcon} holds, $\sv{v}{\omega}$ in  \eqref{eq:SVoperator} reduces to 
\begin{equation}\label{eq:SVDGequal}
    \sv{v}{\omega} = \beta \left( \sum_i \int_{I_i}v\omega_x \mathrm{d}x + \sum_i v_{i+\hf}^-\jump{\omega}_{i+\hf} \right)\quad \forall v,\omega \in \mathbb{V}^k.
\end{equation}
Note that \eqref{eq:SVDGequal} generally {\it cannot} be established if $v \not \in \mathbb{V}^k$. 

\section{Stability and optimal error estimates}
The DG representations of the RKSV and OESV methods motivate us to obtain a framework for analyzing 
 the stability and the optimal error estimates of the fully-discrete RKSV and OESV schemes. Consider the linear advection equation \eqref{eq:1Dlinear} with periodic boundary conditions. 

We will use $C$ to denote a non-negative constant independent of the time step size $\tau$ and the mesh size $h$. Without special marks or explicit statements, $C$ can take different values at different places. The inverse inequalities 
\begin{subequations} 
	\begin{align}
		\nm{\omega_x}_{L^{2}(I_i)}\leq& Ch_i^{-1}\nm{\omega}_{L^2(I_i)},\quad \forall \omega \in \mathbb{P}^k(I_i), \label{eq:normequiv-1}\\  
		\nm{\omega}_{L^\infty(I_i)}\leq& Ch_i^{-\frac{1}{2}}\nm{\omega}_{L^2(I_i)},\quad \forall \omega \in \mathbb{P}^k(I_i),\label{eq:normequiv}
	\end{align}
\end{subequations}
 imply 
\begin{equation}\label{eq:inequalityH}
    \| \sv{v}{\omega} \| \leq Ch^{-1} \nm{v}\nm{\omega} \leq Ch^{-1} \nm{v}_* \nm{\omega}_*.
\end{equation}

\subsection{Main results}
In our following analysis, the stability is established in terms of the energy norm $\nm{\cdot}_*$. 
\begin{theorem}[General stability of RKSV scheme]\label{thm:stability-RKSV}
    Assume that the upwind condition \eqref{eq:UWcon} holds. There exists $\gamma \geq 1$ and $C_{\mathrm{CFL}}^*>0$ such that, under the time step constraint $\frac{\tau}{h^\gamma} \leq C_{\mathrm{CFL}}^* $, the numerical solution of the RKSV scheme for \eqref{eq:1Dlinear} satisfies
    \begin{equation}\label{eq:stability-RKSV}
	\nm{u_h^{n+1}}_*^2\leq \left(1+C_{s}^*\tau\right)\nm{u_h^n}_*^2 + \hat{C}_{s}^* \tau\sum_{0\leq \ell<s}\nm{g^{n,\ell}}^2.
    \end{equation}
    Here $C_{s}^*$ and $\hat{C}_{s}^*$ are non-negative constants independent of $n$.
\end{theorem}

The proof of \Cref{thm:stability-RKSV} will be presented in \Cref{proof:stability-RKSV}. 

\begin{remark}[General stability of OESV scheme]\label{rmk:stability-OESV}
    Assume that in \eqref{eq:linearRK}, $c_{\ell \kappa} \geq 0$, $d_{\ell \kappa} \geq 0$, and $d_{\ell \kappa} = 0$ when $c_{\ell \kappa}=0$, for all $\ell$ and $\kappa$. Under the upwind condition \eqref{eq:UWcon}, one can verify that \eqref{eq:stability-RKSV} holds for the numerical solutions of the OESV scheme \eqref{eq:linearOESV} when $\frac{\tau}{h^2} \leq C_{\mathrm{CFL}}^* $. See \Cref{sec:proof11} for details. 
\end{remark}

\begin{remark}\label{rmk:UWcon}
    The semi-discrete SV method for \eqref{eq:1Dlinear} with $g \equiv 0$ can be written as
    \begin{equation}\label{eq:linearsemiSV}
        \ip{\frac{\mathrm{d}}{\mathrm{d}t} u_h}{\omega}_* = \sv{u_h}{\omega}\quad \forall \omega \in \mathbb{V}^k.
    \end{equation}
    Since $\frac{1}{2}\frac{\mathrm{d}}{\mathrm{d}t} \nm{u_h}_*^2 = \sv{u_h}{u_h} = -\frac{\jump{u_h}^2}{2} \leq 0$ if and only if the upwind condition \eqref{eq:UWcon} is satisfied, the upwind condition actually ensures the stability in the energy norm $\nm{\cdot}_*$.
\end{remark}

The energy norm is equivalent to the $L^2$-norm in the sense of \eqref{eq:propipstar_2}.  We can establish the optimal error estimates in the standard $L^2$-norm $\nm{\cdot}$ as follows. 
\begin{theorem}[Optimal error estimates for OESV scheme]\label{thm:err-OESV}
		Consider the 1D $\mathbb P^k$-based OESV method on quasi-uniform meshes coupled with a $r$-th order RK method with $k\ge 1$ and $r \ge 2$. 
    Assume that the corresponding RKSV scheme (without the OE procedure) is stable in the sense of \eqref{eq:stability-RKSV}. If the exact solution $U(x,t)$ for \eqref{eq:1Dlinear} is sufficiently smooth and  $\nm{u_h^0-u(\cdot,0)}\leq Ch^{k+1}$, then the OESV scheme for \eqref{eq:1Dlinear} with $g \equiv 0$ admits the optimal error estimate:
	\begin{equation}\label{eq:err-OESV}
		\max_{0\leq n\leq\nn}\nm{u_h^n - U(\cdot,n\tau)} \leq C\left(h^{k+1} + \tau^r\right),
	\end{equation}
	whenever $\frac{\tau}{h^\gamma} \leq C_{\mathrm{CFL}}^* $ and $h \leq h_{**}$. Here, $\nn$ is the final time step with $T = \nn \tau$, and $h_{**}$ is a sufficiently small positive constant. 
\end{theorem}

The proof of \Cref{thm:err-OESV} is rather technical and will be given in \Cref{proof:err-OESV}.

\begin{remark}
	The proof of \cref{thm:err-OESV} implies that the above optimal error estimate also holds for the RKSV scheme \eqref{eq:linearRKSV} when $\frac{\tau}{h^\gamma} \leq C_{\mathrm{CFL}}^* $. The assumptions that $k\ge 1$, $r \ge 2$, and $h \leq h_{**}$ are not required for the optimal error estimate of the RKSV scheme without the OE procedure.
\end{remark}

\subsection{Proof of \cref{thm:stability-RKSV}} \label{proof:stability-RKSV}
Our proof of \cref{thm:stability-RKSV} will be based on the energy analysis techniques introduced for the RKDG scheme in \cite{Xu2020SuperconvergenceAO, Xu2020ErrorEO, Xu2019TheLS}. This process consists of three main steps:
\begin{enumerate}
    \item Establish a temporal difference operator that effectively bridges the temporal discretization and spatial discretization.
    \item Formulate an energy equality, which can then be used to derive an energy inequality.
    \item Estimate the source terms involved in the analysis.
\end{enumerate}
To deal with the inner product $\ip{\cdot}{\cdot}_*$ involved in the RKSV scheme \eqref{eq:linearRKSV}, our stability analysis will be conducted within the framework proposed by Sun and Shu in \cite{Sun2018StrongSO}.

\subsubsection{Temporal difference operator}
%For the SV method, the temporal difference operator should be defined to represent \eqref{eq:SVoperator}:
\begin{definition}\label{def:SV-td}
    We introduce the temporal difference operator $\svtd:\mathbb{V}^k \to \mathbb{V}^k$ defined by
    \begin{equation}\label{eq:SV-td}
        \ip{\svtd v}{\omega}_* = \tau \sv{v}{\omega}\quad \forall \omega \in \mathbb{V}^k.
    \end{equation}
\end{definition}
%With the temporal difference operator $\svtd$, we can show that:
\begin{proposition}\label{prop:SV-td}
    There exists coefficients $\{ \alpha_\kappa \}$ with $\alpha_0=1$ and $\{ \alpha_{\ell, \kappa } \}$ such that the RKSV scheme \eqref{eq:linearRKSV} can be rewritten as
    \begin{equation}\label{eq:SV-RKform}
        u_h^{n+1} = \sum_{\kappa = 0}^{s} \alpha_\kappa \svtde^\kappa u_h^n + \tau \sum_{0 \leq \ell < s} \sum_{\kappa = 0}^{s-1-\ell} \alpha_{\ell, \kappa }\svtde^\kappa G^{n,\ell,*}.
    \end{equation}
    The source term $G^{n,\ell,*} \in \mathbb{V}^k$ is the projection defined as follows
    \begin{equation}\label{eq:SV-G}
        \ip{G^{n,\ell,*}}{\omega}_* = \sum_{0\leq \mykappa \leq \ell} d_{\ell \mykappa} \ip{g^{n,\mykappa}}{M^* \omega}\quad \forall \omega \in \mathbb{V}^k.
    \end{equation}
\end{proposition}
\begin{proof}
    By \eqref{eq:linearRKSV} and \cref{def:SV-td}, we have 
    \begin{equation*}
        \ip{u_h^{n,\ell+1}}{\omega}_* = \ip{ \sum_{0\leq \mykappa \leq \ell} \left( c_{\ell \mykappa}u_h^{n,\mykappa} + d_{\ell \mykappa}\svtd u_h^{n,\mykappa} \right)}{\omega}_* + \tau \sum_{0\leq \mykappa \leq \ell} d_{\ell \mykappa}\ip{g^{n,\mykappa}}{M^* \omega},
    \end{equation*}
    for every $\omega \in \mathbb{V}^k$. Since $u_h^{n,0}=u_h^n$, we can verify \eqref{eq:SV-RKform} by induction.
\end{proof}

When the upwind condition \eqref{eq:UWcon} holds, $\svtd$ satisfies several the following  properties, which are crucial to the stability of the RKSV scheme \eqref{eq:linearRKSV}. 
\begin{proposition}[Properties of $\svtd$]\label{prop:SVtd-prop}
    If the upwind condition \eqref{eq:UWcon} is satisfied, the temporal difference operator $\svtd$ satisfies the following properties:
    \begin{enumerate}
        \item The skew-symmetric property:
        \begin{equation}\label{eq:SV-sksymprop}
            \ip{\svtd v}{\omega}_* + \ip{v}{\svtd \omega}_* = -\tau \ipjump{v}{\omega}\quad \forall v, \omega \in \mathbb{V}^k.
        \end{equation}
        \item The non-positive property: for any symmetric semi-positive definite matrix $\{ \theta_{\kappa, \kappa_0} \}$, if $\{ v_{\kappa} \}\subset \mathbb{V}^k$, then
        \begin{equation}\label{eq:SV-nonpprop}
            \sum_{\kappa, \kappa_0} \theta_{\kappa, \kappa_0} \ip{\svtd v_\kappa}{v_{\kappa_0}}_* = -\frac{1}{2}\sum_{\kappa, \kappa_0} \theta_{\kappa, \kappa_0}\ipjump{v_\kappa}{v_{\kappa_0}} \leq 0.
        \end{equation}
        \item The weak boundedness:
        \begin{equation}\label{eq:SV-wkbdd}
            \ip{\svtd v}{\omega}_* \leq C \frac{\tau}{h} \nm{v}_* \nm{\omega}_*\quad \forall v, \omega \in \mathbb{V}^k.
        \end{equation}
    \end{enumerate}
\end{proposition}
\begin{proof}
    The weak boundedness follows from \eqref{eq:inequalityH}. Since \eqref{eq:SVDGequal} holds under the upwind condition, the other two properties can be proven by following the RKDG analysis in \cite[Section 3.1.2]{Xu2019TheLS}.
\end{proof}

\subsubsection{Energy equality and inequality}
Under the framework of Sun and Shu \cite{Sun2018StrongSO}, we introduce the notation
\begin{equation*}\label{eq:Rsv}
    R_s^*(v) := \sum_{\kappa = 0}^{s} \alpha_\kappa \svtde^\kappa v\quad \forall v \in \mathbb{V}^k,
\end{equation*}
and derive the following energy equality for the RKSV scheme. 
\begin{proposition}[Energy equality]\label{prop:SV-Eiequal}
    There exists constants $\{ \beta_\kappa \}$ and $\{ \beta_{\kappa, \kappa_0} \}$ with $\beta_{\kappa,\kappa_0} = \beta_{\kappa_0,\kappa}$ such that for any $v \in \mathbb{V}^k$,
    \begin{equation}\label{eq:SV-Eequal}
        \nm{R_s^*(v)}_*^2 = \nm{v}_*^2 + \sum_{\kappa = 1}^{s}\beta_\kappa \nm{\svtde^\kappa v}_*^2 - \tau \sum_{0\leq \kappa,\kappa_0 <s} \beta_{\kappa,\kappa_0} \ipjump{\svtde^\kappa v}{\svtde^{\kappa_0} v},
    \end{equation}
where the semi-inner product $\ipjump{\cdot}{\cdot}$ is defined in \eqref{jump}. 
\end{proposition}
\begin{proof}
    The definition of $R_s^*$ yields that for any $v \in \mathbb{V}^k$,
    \begin{equation}\label{eq:SV-Eequal_1}
        \nm{R_s^*(v)}_*^2 = \sum_{0 \leq \kappa,\kappa_0 \leq s}\alpha_\kappa \alpha_{\kappa_0} \ip{\svtde^\kappa v}{\svtde^{\kappa_0} v}_*.
    \end{equation}
    Since $\alpha_0=1$ (\cref{prop:SV-td}), we obtain \eqref{eq:SV-Eequal} by inductively applying the skew-symmetric property \eqref{eq:SV-sksymprop} to the right hand side of  equality \eqref{eq:SV-Eequal_1}.
\end{proof}

With \cref{prop:SV-Eiequal} and the properties of $\svtd$ concluded in \cref{prop:SVtd-prop}, we can adopt the RKDG analysis approach in \cite[Section 3.4]{Xu2019TheLS} to obtain the following energy inequality for the RKSV method. 
\begin{corollary}[Energy inequality]\label{cor:SV-Eiequal}
    There exists $\gamma_0 \geq 2$ such that:
   \begin{equation}\label{eq:SV-Eiequal}
        \nm{R_s(v)}_*^2 \leq \left( 1+C \left(\frac{\tau}{h}\right)^{\gamma_0} \right)\nm{v}_*^2 \quad \forall v \in \mathbb{V}^k,
    \end{equation}
    whenever $\frac{\tau}{h}\leq C_{\mathrm{CFL}}^*$ for some $C_{\mathrm{CFL}}^*>0$. Here, the constant $C$ can be $0$.
\end{corollary}

Thanks to the established relation between the RKDG and RKSV methods, the proof of \Cref{cor:SV-Eiequal} directly follows from the RKDG results in \cite{Xu2019TheLS} and is thus omitted. 
The energy inequality \eqref{eq:SV-Eiequal} implies that the RKSV scheme is stable with respect to the energy norm $\nm{\cdot}_*$ in a weak sense as follows.
\begin{theorem}[Weak($\gamma_0$) stability of RKSV scheme]
    The numerical solution of the RKSV scheme \eqref{eq:linearRKSV} with $g \equiv 0$ satisfies
    \begin{equation}
        \nm{u_h^{n+1}}_*^2 \leq \left( 1+C \left(\frac{\tau}{h}\right)^{\gamma_0} \right)\nm{u_h^n}_*^2,
    \end{equation}
    for some $\gamma_0 \geq 2$ whenever $\frac{\tau}{h}\leq C_{\mathrm{CFL}}^*$. Here, $C$ is independent of $n$ and can be $0$.
\end{theorem}

\begin{remark}\label{rmk:DGSVsameresult}
    From the proofs of \cref{prop:SV-td,prop:SV-Eiequal}, it is straightforward to observe that the coefficients $\{ \alpha_\kappa \}$ in \eqref{eq:SV-RKform} and $\{ \beta_\kappa \}$, $\{ \beta_{\kappa, \kappa_0} \}$ in \eqref{eq:SV-Eequal} \textit{are determined only by the RK time discretization}. Hence, if $g \equiv 0$, replacing the standard $L^2$-norm $\nm{\cdot}$ by $\nm{\cdot}_*$, all the existing stability results of the RKDG schemes in  \cite{Xu2020SuperconvergenceAO,Xu2020ErrorEO,Xu2019TheLS} are naturally extensible to the  RKSV schemes. For example, we have \Cref{monotone}.     For more stability results of the RKDG schemes, see \cite{Xu2019TheLS,Xu2020ErrorEO,Xu2020SuperconvergenceAO}.

\end{remark}   

 \begin{theorem}[Monotonicity stability of RKSV scheme]\label{monotone}
Suppose that the upwind condition \eqref{eq:UWcon} holds. Consider the RKSV scheme \eqref{eq:linearRKSV} coupled with an $r$th-order $r$-stage RK method and $r \equiv 3 \ (\rm{mod} \ 4)$. Whenever $\frac{\tau}{h}\leq C_{\mathrm{CFL}}^*$, the numerical solution of the RKSV scheme \eqref{eq:linearRKSV} with $g \equiv 0$ satisfies
\begin{equation}
	\nm{u_h^{n+1}}_*^2 \leq \nm{u_h^n}_*^2.
\end{equation}
\end{theorem}

\subsubsection{Estimates for source terms}
\begin{proposition}\label{prop:SV-G-est}
    There exists a constant $C>0$ independent of $n$ such that 
    \begin{equation}\label{eq:SV-G-est}
        \nm{G^{n,\ell,*}}_*^2 \leq C\sum_{\kappa=0}^\ell \nm{g^{n,\ell}}^2,\quad 0\leq \ell \leq s-1.
    \end{equation}
\end{proposition}
\begin{proof}
    According to \eqref{eq:SV-G} and the Cauchy--Schwarz inequality, we derive  
    \begin{equation*}
        \frac{\nm{G^{n,\ell,*}}_*^2}{\sqrt{\sum_{0\leq \mykappa \leq \ell} \nm{G^{n,\ell,*}}^2}} = \frac{\sum_{0\leq \mykappa \leq \ell} d_{\ell \mykappa} \ip{g^{n,\mykappa}}{M^* G^{n,\ell,*}}}{\sqrt{\sum_{0\leq \mykappa \leq \ell} \nm{G^{n,\ell,*}}^2}} \leq \nm{M^*} \sqrt{\sum_{0\leq \mykappa \leq \ell} d_{\ell \mykappa}^2\nm{g^{n,\mykappa}}^2} .
    \end{equation*}
    We then obtain the estimate \eqref{eq:SV-G-est} by using $\nm{M^*}\leq C$ (\cref{prop:Mstar}) and $\nm{G^{n,\ell,*}} \leq C\nm{G^{n,\ell,*}}_*$ (\cref{prop:ipstar}).
\end{proof}

\subsubsection{Proof of \eqref{eq:stability-RKSV}}
Without the loss of generality, we assume that $h \leq 1$. Let $\gamma=\frac{\gamma_0}{\gamma_0-1} \geq 1$. When $\frac{\tau}{h^\gamma}\leq C_{\mathrm{CFL}}^*$, we have $\frac{\tau}{h}=\frac{\tau}{h^\gamma}h^{\gamma-1} \leq C_{\mathrm{CFL}}^*$, and then obtain $\nm{R_s(u_h^n)}_*^2 \leq \left( 1+C \tau \right)\nm{u_h^n}_*^2$ by \cref{cor:SV-Eiequal}. Using \eqref{eq:SV-RKform} and \cref{prop:SV-G-est} gives  
\begin{equation*}
    \begin{aligned}
        \nm{u_h^{n+1}}_*^2 & = \nm{R_s(u_h^n) + \tau\sum_{\kappa, \ell} \alpha_{\ell, \kappa }\svtde^\kappa G^{n,\ell,*}}_*^2 \\
        & \leq (1+\tau)\nm{R_s(u_h^n)}_*^2 + C(1+T)\tau \sum_{\kappa} \nm{\svtd}_*^\kappa \cdot \sum_{\ell} \nm{G^{n,\ell,*}}_*^2 \\
        & \leq \left( 1+\left( C+T \right)\tau \right)\nm{u_h^n}_*^2 + C\tau \sum_{\kappa} \nm{\svtd}_*^\kappa \cdot \sum_{\ell} \nm{g^{n,\ell}}^2 ,
    \end{aligned}
\end{equation*}
since $\tau \leq T$. Consequently, we obtain \eqref{eq:stability-RKSV}, by noting that \eqref{eq:SV-wkbdd} implies 
\begin{equation*}%\label{eq:05192124}
    \nm{\svtd}_* = \sup_{v \not \equiv 0} \frac{\nm{\svtd v}_*}{\nm{v}_*} \leq C \frac{\tau}{h} \leq C.
\end{equation*}

\subsection{Proof of \cref{thm:err-OESV}}\label{proof:err-OESV}
Based on the established connection between the DG and SV methods, we will prove \cref{thm:err-OESV} using the error estimation techniques for the RKDG scheme in \cite{Xu2020SuperconvergenceAO,Xu2020ErrorEO}. The proof involves three key components:
\begin{enumerate}
    \item Reference functions of the exact solution to measure the error at each RK stages.
    \item A projection with certain approximation and superconvergence properties.
    \item An error decomposition based on the aforementioned projection.
\end{enumerate}

\subsubsection{Reference functions}
The reference functions of the exact solution $U$ is defined according to \cite{Xu2020SuperconvergenceAO,Xu2020ErrorEO} as follows. 
\begin{definition}[Reference functions]\label{def:reffun}
    For $0\leq n\leq\nn$, denote $U^n = U(\cdot,n\tau)$. The reference functions are  recursively defined as 
    \begin{equation}\label{eq:reffun}
        \begin{aligned}
            & U^{n,0}=U^{n}, \\
            & U^{n,\ell+1}=\sum_{0\leq \kappa \leq \ell}\left( c_{\ell \kappa}U^{n,\kappa}-\tau d_{\ell \kappa}\beta U_x^{n,\kappa} \right),\quad \ell = 0,...,s-2, \\
            & U^{n,s}=U^{n+1}.
        \end{aligned}
    \end{equation}
\end{definition}
To derive error estimate\eqref{eq:err-OESV}, we require the following formulation for $\{ U^{n,\kappa} \}$. 
\begin{proposition}[Formulation for reference functions]\label{prop:SV-reffun}
    The reference functions $\{ U^{n,\kappa} \}$ satisfy
    \begin{equation}\label{eq:SV-reffun}
        \ip{U^{n,\ell+1}}{M^* \omega} = \sum_{0 \leq \kappa \leq \ell}\left( c_{\ell \kappa}\ip{U^{n,\kappa}}{M^* \omega} + \tau d_{\ell \kappa}\sv{U^{n,\kappa}}{\omega} \right) + \tau\ip{\rho^{n,\ell+1}}{M^* \omega},
    \end{equation}
\end{proposition}
where the local truncation error $\rho^{n,\ell+1}$ is defined as 
\begin{equation}\label{eq:lte}
    \rho^{n,\ell+1}(x) := \begin{cases}
        \frac1{\tau} \left({U(x,(n+1)\tau) - \sum_{0\leq \kappa \leq \ell}\left( c_{\ell \kappa}U^{n,\kappa}-\tau d_{\ell \kappa}\beta U_x^{n,\kappa} \right)}\right),~ &\mbox{if}~ \ell=s-1;\\
        0,~ &\mbox{otherwise}.
    \end{cases}
\end{equation}
\begin{proof}
    Since the exact solution  $U$ is assumed to be sufficiently smooth, $U^{n,\kappa}$ is continuous. Similar to the proof of \cref{thm:dgr-1DsemiSV}, we can show that 
    \begin{equation}\label{eq:proofreffun}
        \ip{-\beta U_x^{n,\kappa}}{M^* \omega} = \sv{U^{n,\kappa}}{\omega}\quad \forall \omega \in \mathbb{V}^k.
    \end{equation}
    We then obtain \eqref{eq:SV-reffun} by \cref{def:reffun} and \eqref{eq:proofreffun}.
\end{proof}

\begin{proposition}[Local truncation error]\label{prop:lte}
    If the exact solution $U$ is sufficiently smooth, the local truncation error $\rho^{n,\ell+1}$ satisfies
    \begin{equation}
        \nm{\rho^{n,\ell+1}}_{L^\infty(\Omega)} \leq C\tau^{r},\quad  0\leq n \leq \nn,~ \ell=0,\dots,s-1 .
    \end{equation}
\end{proposition}

\Cref{prop:lte} follows from the Taylor expansion and its proof is omitted.

\subsubsection{A projection operator}
Inspired by the semi-discrete analysis in \cite{Cao2021AnalysisOS}, we introduce the following projection operator $P^*$.
\begin{definition}[Projection $P^*$]\label{def:Pstar}
    For $v \in  L^2(\Omega)$, the projection $P^* v \in \mathbb{V}^k$ of $v$ is the piecewise-interpolating polynomial such that in each $I_i$,
    \begin{equation*}
       \left( P^* v \right)(x_{i,1}) = v(x_{i,1}), \dots , \left( P^* v \right)(x_{i,k}) = v(x_{i,k}),\ \left( P^* v \right)(x_{i+\hf}^-) = v(x_{i+\hf}^-).
    \end{equation*}
\end{definition}
The approximation and superconvergence properties of $P^*$ are demonstrated as follows.  
\begin{proposition}[Properties of $P^*$]\label{prop:SV-proj}
    The projection operator $P^*$ satisfies: 
    \begin{enumerate}
        \item There exists a constant $C$ independent of $v$ such that 
        \begin{equation}\label{eq:SV-approxprop}
            \nm{\eta^*} \leq Ch^{k+1}\nm{\partial_x^{k+1}v}_{L^\infty(\Omega)}.
        \end{equation}

        \item If $v$ is continuous, then 
        \begin{equation}\label{eq:SV-supconprop}
            \sv{\eta^*}{\omega} = 0\quad \forall \omega \in \mathbb{V}^{k}.
        \end{equation}
    \end{enumerate}
    
    In both \eqref{eq:SV-approxprop} and \eqref{eq:SV-supconprop}, $\eta^* = v - P^*v$.
\end{proposition}
\begin{proof}
    The property \eqref{eq:SV-approxprop} holds since $P^*$ is a piecewise polynomial interpolation operator. The property \eqref{eq:SV-supconprop} follows from the fact $\eta^* (x_{i+\hf}^-) = 0$ for all $i$. 
\end{proof}

\subsubsection{Error decomposition}
At each RK stage and each time step, we consider the following error decomposition based on the projection $P^*$.
\begin{definition}[Error decomposition]\label{def:SV-errdecompose}
    \begin{equation}
        u_h^n - U^n = \xi^{n,*} - \eta^{n,*} \quand u_h^{n,\kappa} - U^{n,\kappa} =  \xi^{n, \mykappa,*} - \eta^{n, \mykappa,*},
    \end{equation}
    where $\{ \xi^{n,*} \}$, $\{ \xi^{n, \mykappa,*} \}$, $\{ \eta^{n,*} \}$, and $\{ \eta^{n, \mykappa,*} \}$ are defined as
    \begin{subequations}\label{eq:SV-errdecompose}
        \begin{align}
           \xi^{n,*} = u_h^n - P^* U^n, \qquad & \xi^{n, \mykappa,*} = u_h^{n,\kappa} - P^* U^{n,\kappa}, \\
           \eta^{n,*} = U^n - P^* U^n, \qquad & \eta^{n, \mykappa,*} = U^{n,\kappa} - P^* U^{n,\kappa}.
        \end{align}
    \end{subequations}
\end{definition}
Inspired by \cite{Peng2023OEDGOD}, we skillfully treat the OESV scheme \eqref{eq:linearOESV} as a RKSV scheme with the ``source'' terms $\{ u_h^{n,\ell+1} - \tilde{u}_h^{n,\ell+1} \}$. By subtracting \eqref{eq:SV-reffun} from \eqref{eq:linearOESV} and using the superconvergence property of $P^*$, we obtain a novel formulation for $\{ \xi^{n,\kappa,*} \}$. 
\begin{proposition}[Formulation for $\{ \xi^{n,\kappa,*} \}$]\label{prop:OESV-xi}
    For the OESV scheme \eqref{eq:linearOESV}, define 
    $$\cF^{n,\ell+1} = \frac{u_h^{n,\ell+1} - \tilde{u}_h^{n,\ell+1}}{\tau}, \quad \cZ^{n,\ell+1,*} = \frac{\eta^{n,\ell+1,*}-\sum_{0 \leq \kappa \leq \ell}c_{\ell \kappa}\eta^{n,\kappa,*}}{\tau}-\rho^{n, \ell+1},$$ 
    then $\{ \xi^{n,\kappa,*} \}$ satisfies 
    \begin{equation}\label{eq:xi-scheme}
        \begin{aligned}
            \ip{\xi^{n,\ell+1,*}}{\omega}_* & = \sum_{0 \leq \mykappa \leq \ell} \left(c_{\ell \mykappa}\ip{\xi^{n,\mykappa,*}}{\omega}_* + \tau d_{\ell \mykappa}\sv{\xi^{n,\mykappa,*}}{\omega}\right) \\
            & \quad  + \tau \ip{\cF^{n,\ell+1}}{M^* \omega} + \tau\ip{\cZ^{n,\ell+1,*}}{M^* \omega},
        \end{aligned}
    \end{equation}
\end{proposition}

	The following proposition will be useful for the estimate of $\{ \cZ^{n,\ell+1,*} \}$. 
\begin{proposition}\label{prop:SV-eta}
    There exists $C>0$ independent of $n$ such that
    \begin{equation}\label{eq:SV-eta-2}
        \nm{\eta^{n,\ell+1,*}-\sum_{0 \leq \kappa \leq \ell}c_{\ell \kappa}\eta^{n,\kappa,*}} \leq C\tau h^{k+1}, \quad 0 \leq \ell \leq s-1.
    \end{equation}
\end{proposition}
\begin{proof}
    Since $\sum_{0 \leq \kappa \leq \ell}c_{\ell \kappa} = 1$, using \eqref{eq:reffun} and \eqref{eq:lte} gives 
    \begin{equation*}
        U^{n,\ell+1}-\sum_{0 \leq \kappa \leq \ell}c_{\ell \kappa}U^{n,\kappa} = \tau \left(\sum_{0 \leq \kappa \leq \ell} \hat{c}_{\ell \kappa}\tau^{\kappa} \partial_x^{\kappa+1} U^n \right) + \tau\rho^{n,\ell+1}
    \end{equation*}
    for some constants $\{ \hat{c}_{\ell \kappa} \}$. By \cref{prop:lte} and the approximation property \eqref{eq:SV-approxprop} of $P^*$, we obtain for $0 \leq \ell \leq s-1$ that
    \begin{equation*}
        \nm{\eta^{n,\ell+1,*}-\sum_{0 \leq \kappa \leq \ell}c_{\ell \kappa}\eta^{n,\kappa,*}} \leq \tau \left( C \sum_{0 \leq \kappa \leq s-1}(T)^\kappa  + C(T)^r \right)h^{k+1}.
    \end{equation*}
The proof is completed. 
\end{proof}

To derive the estimate for $\nm{\cF^{n,\ell+1}}$, we introduce 
\begin{equation*}
     \tilde{\xi}^{n,\mykappa,*} = \tilde{u}_h^{n,\kappa} - P^* U^{n,\kappa},
\end{equation*}
for $0\leq n \leq \nn$ and $0\leq \mykappa \leq s-1$. 

\begin{proposition}\label{prop:OE-1}
    If $\nm{\tilde{\xi}^{n,\ell+1,*}}_{L^\infty (\Omega)} \leq h$ and $ h \leq h_*$ for some sufficiently small $h_*>0$, then $\nm{\cF^{n,\ell+1}} \leq C\left( \nm{\tilde{\xi}^{n,\ell+1,*}}_* + h^{k+1} \right)$. Here $C$ is independent of $n$.
\end{proposition}

\begin{proposition}\label{prop:xi-tilde}
    For $0 \leq \ell \leq s-1$, there exists a constant $C>0$ such that:
    \begin{equation}
        \nm{\tilde{\xi}^{n,\ell+1,*}}_* \leq C \left( \sum_{0 \leq \kappa \leq \ell} \nm{\xi^{n,\kappa,*}}_*  \right) + C\tau(h^{k+1}+\tau^r)\quad \forall n. 
    \end{equation}
\end{proposition}

Due to the established connection between the OESV and OEDG methods, 
the proofs of \Cref{prop:OE-1,prop:xi-tilde} directly follow from the analysis for OEDG schemes in \cite{Peng2023OEDGOD}[Section 4.4.2] and are thus omitted here. 

Now we derive our estimate for $\nm{\cF^{n,\ell+1}}$. 
\begin{proposition}\label{prop:OE}
    If $\nm{\tilde{\xi}^{n,\kappa,*}}_{L^\infty (\Omega)} \leq h$ for $1 \leq \kappa \leq \ell+1$ and $ h \leq h_*$, then 
    \begin{equation}\label{eq:propOE}
        \nm{\cF^{n,\ell+1}} \leq C \left( \nm{\xi^{n,*}}_* + h^{k+1} + \tau^r \right),
    \end{equation}
    where the constant $C>0$ is independent of $n$.
\end{proposition}
\begin{proof}
   Since $\nm{\tilde{\xi}^{n,\ell + 1,*}}_{L^\infty (\Omega)} \leq h$, we have $\nm{\cF^{n,\ell+1}} \leq C\left( \nm{\tilde{\xi}^{n,\ell + 1,*}}_* + h^{k+1} \right)$ by \cref{prop:OE-1}. According to \cref{prop:xi-tilde}, to prove \eqref{eq:propOE}, it suffices to show that
   \begin{equation}\label{eq:lemmaOE-1}
       \nm{\xi^{n,\kappa,*}}_* \leq C_{\kappa} \left( \nm{\xi^{n,*}}_* + h^{k+1} + \tau^r \right) \quad 1 \leq \kappa \leq \ell.
   \end{equation}
   To prove \eqref{eq:lemmaOE-1},  we first notice that
   \begin{equation*}\label{eq:lemmaOE-2}
       \begin{aligned}
           \nm{\xi^{n,1,*}}_* \leq \nm{\tilde{\xi}^{n,1,*}}_* + \tau \nm{\cF^{n,1}}_* & \leq C \left( \nm{\tilde{\xi}^{n,1,*}}_* + h^{k+1} + \tau^r \right) \\
           & \leq C_1 \left( \nm{\xi^{n,*}}_* + h^{k+1} + \tau^r \right),
       \end{aligned}
   \end{equation*}
   where we have used $\nm{\cF^{n,1}}_* \leq C \nm{\cF^{n,1}}$ (\cref{prop:ipstar}). If we assume that \eqref{eq:lemmaOE-1} holds for $1 \leq \kappa \leq \ell_0$, then 
   \begin{equation*}
       \begin{aligned}
           \nm{\xi^{n,\ell_0 + 1,*}}_* & \leq \nm{\tilde{\xi}^{n,\ell_0 + 1,*}}_* + \tau \nm{\cF^{n,\ell_0 + 1}}_* \leq C \left( \nm{\tilde{\xi}^{n,\ell_0 + 1,*}}_* + h^{k+1} + \tau^r \right) \\
           \leq C & \left( \sum_{0 \leq \kappa \leq \ell_0} \nm{\xi^{n,\kappa,*}}_* + h^{k+1} + \tau^r \right) \leq C_{\ell_0 + 1} \left( \nm{\xi^{n,*}}_* + h^{k+1} + \tau^r \right).
       \end{aligned}
   \end{equation*}
   Hence, \eqref{eq:lemmaOE-1} holds for all $1 \leq \kappa \leq \ell$ by mathematical induction. 
\end{proof}

\subsubsection{Proof of \eqref{eq:err-OESV}}
According to the approximation property of $P^*$, we have
\begin{equation}\label{eq:SV-eta-1}
    \nm{\eta^{n,*}} \leq Ch^{k+1}, \quad 0 \leq n \leq \nn.
\end{equation}
Hence, to prove \eqref{eq:err-OESV}, we only need to show the following proposition.  
\begin{proposition}\label{prop:xi-final}
   Under the assumption of \Cref{thm:err-OESV}, there exists two constants $C_{**}>0$ and $h_{**}>0$ such that
    \begin{equation}\label{eq:xi-final}
        \max_{0\leq n\leq\nn} \nm{\xi^{n,*}}_*^2 \leq C_{**}\left( h^{2k+2} + \tau^{2r} \right),
    \end{equation}
    whenever $\frac{\tau}{h^\gamma} \leq C_{\mathrm{CFL}}^{*}$ and $h \leq h_{**}$.
\end{proposition}

%{\bf Inductions on $\nm{\xi^{n,*}}_*$.} 
\begin{proof}  
We first consider the case that $n=0$. Since we assume that $\nm{u_h^0 - U^0} \leq Ch^{k+1}$, the approximation property of $P^*$ implies there exists $C_*>0$ such that
\begin{equation}\label{eq:05152024}
    \nm{\xi^{0,*}}_*^2 \leq C\left( \nm{u_h^0 - U^0}^2 + \nm{\eta^{0,*}}^2  \right) \leq C_*\left( h^{2k+2} + \tau^{2r} \right).
\end{equation}
Under the time step constraint $\frac{\tau}{h^{\gamma}} \leq C_{\mathrm{CFL}}^{*}$,   \cref{prop:xi-tilde} yields 
\begin{equation*}\label{eq:error-2}
    \nm{\tilde{\xi}^{0,1,*}} \leq C\nm{\tilde{\xi}^{0,1,*}}_* \leq C \left( h^{k+1} + \tau^{r} \right) \leq C \left( h^{k+1} + h^{\gamma r} \right).
\end{equation*}
Applying the inverse inequality \eqref{eq:normequiv} gives $\nm{\tilde{\xi}^{0,1,*}}_{L^\infty(\Omega)} \leq C \left( h^{k+\hf} + h^{\gamma r - \hf} \right)$. Since $k>\hf$ and $\gamma r \geq r > \frac{3}{2}$, there exists a constant  $ h_0^1 \in (0, h_*]$ such that 
\begin{equation}\label{eq:error-4}
    h \leq h_0^1 ~ \Rightarrow ~ \nm{\tilde{\xi}^{0,1,*}}_{L^\infty(\Omega)} \leq h.
\end{equation}
Suppose that there exists $h_0^\ell \in (0, h_*]$ such that
\begin{equation*}\label{eq:error-5}
    h \leq h_0^\ell ~ \Rightarrow ~ \nm{\tilde{\xi}^{0,\kappa,*}}_{L^\infty(\Omega)} \leq h \quad 1 \leq \kappa \leq \ell.
\end{equation*}
When $h\leq h_0^\ell$,  using  \cref{prop:xi-tilde} and \eqref{eq:lemmaOE-1} gives 
\begin{equation*}\label{eq:error-6}
    \begin{aligned}
        \nm{\tilde{\xi}^{0,\ell+1,*}} & \leq \nm{\tilde{\xi}^{0,\ell+1,*}}_* \leq C \left( \sum_{0 \leq \kappa \leq \ell} \nm{\xi^{0,\kappa,*}}_* + h^{k+1} + \tau^r \right) \\
        & \leq C \left(\nm{\xi^{0,*}}_* + h^{k+1} + \tau^r \right) \leq C \left( h^{k+1} + \tau^{r} \right) \leq C \left( h^{k+1} + h^{\gamma r} \right). 
    \end{aligned}
\end{equation*}
Hence, similar to the proof of \eqref{eq:error-4}, we can also find $h_0^{\ell + 1} \in (0, h_0^\ell]$ such that
\begin{equation}\label{eq:error-7}
    h \leq h_0^{\ell+1} ~ \Rightarrow ~ \nm{\tilde{\xi}^{0,\kappa,*}}_{L^\infty(\Omega)} \leq h \quad 1 \leq \kappa \leq \ell+1.
\end{equation}
Combining \eqref{eq:error-4} with \eqref{eq:error-7}, we can conclude by the induction hypothesis that there exists a constant $h_0^s \in (0, h_*]$ such that
\begin{equation}\label{eq:h_0^s}
    h \leq h_0^s ~ \Rightarrow ~ \nm{\tilde{\xi}^{0,\ell+1,*}}_{L^\infty(\Omega)} \leq h \quad 0 \leq \ell \leq s-1.
\end{equation}
It then follows from  \cref{prop:OE} and \eqref{eq:05152024} that $\nm{\cF^{0,\ell+1}} \leq C \left( h^{k+1} + \tau^r \right)$ for $0 \leq \ell \leq s-1$ if $h \leq h_0^s$. Notice that \cref{prop:lte,prop:SV-eta} imply 
\begin{equation}\label{eq:05151721}
    \sum_{0\leq \ell<s}\nm{\cZ^{n,\ell+1,*}}^2 \leq C\left( h^{2k+2} + \tau^{2r} \right) \quad 0 \leq n \leq \nn.
\end{equation}
Since we assume that the corresponding RKSV scheme is stable, using \eqref{eq:stability-RKSV} and the discrete Gronwall inequality, we obtain 
\begin{equation}\label{eq:error-8}
    \begin{aligned}
        \nm{\xi^{1,*}}_*^2 & \leq (1+C_{s}^* \tau) \nm{\xi^{0,*}}_*^2 + \hat{C}_{s}^*\tau(h^{2k+2} + \tau^{2r}) \\
        & \leq e^{C_{s}^* T}\left( \nm{\xi^{0,*}}_*^2 + \hat{C}_{s}^{*} T(h^{2k+2} + \tau^{2r}) \right) \leq C_{**}\left( h^{2k+2} + \tau^{2r} \right),
    \end{aligned}
\end{equation}
where $C_{**} := e^{C_{s}^* T}(C_* + \hat{C}_{s}^{*} T) \geq C_*$.

Next, for $n=1$, adopting the above procedure to $\{ \tilde{\xi}^{1,\kappa,*} \}$, we can derive by \eqref{eq:error-8} that there exists a constant $h_{**} \in (0, h_0^s]$ such that
\begin{equation}\label{eq:error-10}
     h \leq h_{**} ~ \Rightarrow ~ \nm{\tilde{\xi}^{1,\ell+1,*}}_{L^\infty(\Omega)} \leq h, \quad 0 \leq \ell \leq s-1.
\end{equation}
Thus, if $h \leq h_{**}$, similar to \eqref{eq:error-8}, we have
\begin{equation}\label{eq:07052024-1}
    \begin{aligned}
        \nm{\xi^{2,*}}_*^2 & \leq (1+C_{s}^* \tau)^{2} \nm{\xi^{0,*}}_*^2 + \hat{C}_{s}^*\tau \sum_{n=0}^{1}(1+C_{s}^* \tau)^n (h^{2k+2} + \tau^{2r}) \\
        & \leq e^{C_{s}^* T}\left( \nm{\xi^{0,*}}_*^2 + \hat{C}_{s}^{*} T(h^{2k+2} + \tau^{2r}) \right) \leq C_{**}\left( h^{2k+2} + \tau^{2r} \right).
    \end{aligned}
\end{equation}
Now suppose that $\nm{\xi^{n,*}}_*^2 \leq C_{**}\left( h^{2k+2} + \tau^{2r} \right)$ for $n=n_0\geq 2$, repeat the similar arguments as the case of $n=1$, we can show that
\begin{equation}\label{eq:error-9}
    h \leq h_{**} ~ \Rightarrow ~ \nm{\tilde{\xi}^{n_0,\ell+1,*}}_{L^\infty(\Omega)} \leq h, \quad 0 \leq \ell \leq s-1.
\end{equation}
As a result, \begin{equation}\label{eq:05152024-1}
    \begin{aligned}
        \nm{\xi^{n_0+1,*}}_*^2 & \leq (1+C_{s}^* \tau)^{n_0+1} \nm{\xi^{0,*}}_*^2 + \hat{C}_{s}^*\tau \sum_{n=0}^{n_0}(1+C_{s}^* \tau)^n (h^{2k+2} + \tau^{2r}) \\
        & \leq e^{C_{s}^* T}\left( \nm{\xi^{0,*}}_*^2 + \hat{C}_{s}^{*} T(h^{2k+2} + \tau^{2r}) \right) \leq C_{**}\left( h^{2k+2} + \tau^{2r} \right).
    \end{aligned}
\end{equation}
Combining \eqref{eq:05152024} with \eqref{eq:07052024-1} and \eqref{eq:05152024-1}, we complete the proof of \cref{prop:xi-final} by the induction hypothesis.
\end{proof}

\subsection{Further discussions on the upwind condition \eqref{eq:UWcon}}
From the above analysis, we have seen that the upwind condition \eqref{eq:UWcon} is important for the stability of the RKSV scheme and the optimal convergence rate of our OESV method. In this subsection, we further discuss this important condition. %We discover that:
\begin{theorem}\label{thm:pik}
    Under \cref{assump:subdivision}, the upwind condition \eqref{eq:UWcon} holds if and only if in each $I_i$, the subdivision points $\{ x_{i,j} \}_{j=1}^{k}$ are the $k$ distinct zeros of $p_i^k$, where 
    \begin{equation}\label{eq:pik}
        p_i^k = L_{i,k}+\frac{c}{k(1-c)}\left( x-x_{i-\hf} \right)L_{i,k}^{'} \quad \mbox{with} \quad c<1.
    \end{equation}
\end{theorem}
\begin{proof}
    Let $x_i=\frac{x_{i-\hf}+x_{i+\hf}}{2}$. The properties of the Legendre polynomials yield
    \begin{subequations}
        \begin{align}
            & (k+1)L_{i,k+1} = \frac{2(2k+1)}{h_i}(x-x_i)L_{i,k} - kL_{i,k-1}, \label{eq:pik1} \\
            & \frac{\left( x-x_{i-\hf} \right) \left( x-x_{i+\hf} \right)}{k}L_{i,k}^{'} = (x-x_i)L_{i,k} - \frac{2}{hi}L_{i,k-1}. \label{eq:pik2}
        \end{align}
    \end{subequations}
    By \eqref{eq:pik1}, we have 
    \begin{equation}\label{eq:pik3}
        \frac{h_i}{2k-1} - Q_i^k(L_{i,k+1}L_{i,k-1})>0 ~ \Leftrightarrow ~ \frac{h_i}{2k-1} - \frac{2}{h_i}Q_i^k\left( (x-x_i)L_{i,k}L_{i,k-1} \right)>0.
    \end{equation}
    According to \eqref{eq:pik3} and the quadrature theory \cite{Brass2011QuadratureTT}, the upwind condition \eqref{eq:UWcon} holds if and only if for every $i$, there exists $a_i \neq 0$ such that 
    \begin{equation*}
        a_i\left( x-x_{i,1} \right)\cdots\left( x-x_{i,k+1} \right) = (x-x_i)L_{i,k} + \frac{(c-1)h_i}{2}L_{i,k} - \frac{ch_i}{2}L_{i,k-1},\quad c<1.
    \end{equation*}
    Thus, by \eqref{eq:pik2} we have $(x-x_i)L_{i,k} + \frac{(c-1)h_i}{2}L_{i,k} - \frac{ch_i}{2}L_{i,k-1} = (1-c)\left( x-x_{i,k+1} \right)p_i^k$. Hence, the upwind condition holds if and only if $\left( x-x_{i,1} \right)\cdots\left( x-x_{i,k} \right) = \frac{p_i^k}{a_i}$.
\end{proof}

\begin{remark}
    The semi-discrete SV schemes based on the zeros of $p_i^k$ are exactly those studied by Lu, Jiang, Shu, and Zhang in \cite{LuAnalysisOA}. They proved in \cite{LuAnalysisOA} that $p_i^k$ has $k$ distinct zeros within the interior of $I_i$ if $c>-\frac{1}{k}$. See \cite{LuAnalysisOA} for more details.
\end{remark}

\section{Numerical tests}
This section presents several 1D and 2D numerical examples to validate the accuracy and effectiveness of the OESV scheme proposed in \cref{sec:OESV}.  
For smooth problems, we couple the $\mathbb{P}^k$ or $\mathbb{Q}^k$-based OESV scheme with a $(k+1)$th-order explicit RK time discretization to verify the $(k+1)$th-order accuracy; for problems with discontinuities, we apply the third-order strong-stability-preserving explicit RK method for time discretization. The Gauss quadrature points are chosen as the subdivision points of the OESV scheme. 
For both $\mathbb{P}^k$ and $\mathbb{Q}^k$-based OESV schemes, we set the CFL number as  $C_{\mathrm{CFL}}=\frac{1}{2k+1}$ unless otherwise stated. More numerical examples are provided in \Cref{app:numex}. 

\subsection{1D and 2D linear advection equations}

\begin{exmp}[smooth problem]\label{ex:1Dadvec}
    This example is used to validate the optimal convergence rate of the OESV schemes for the  advection equation $u_t+u_x=0$  on $\Omega = [0,1]$ with periodic boundary conditions. 
    The initial condition is $u_0(x)=\sin^2(2 \pi x)$. The numerical errors and the corresponding convergence rates for the $\mathbb{P}^k$-based OESV scheme at time $t=1.1$ are listed in \cref{table:test1}. We observe that the $\mathbb P^k$-based OESV scheme exhibits an optimal $(k+1)$th-order convergence rate. As also observed in the OEDG method \cite{Peng2023OEDGOD}, the error is dominated by the high-order damping effect of the OE procedure on the coarser meshes, yielding a rate higher than $k+1$ for smaller $N_x$. 

              \begin{table}[!tbh]
    	\centering
    	% h-here,t-top,b-bottom,优先级依次下降
    	\begin{center}
    		% 居中 
    		\caption{Errors and convergence rates of $\mathbb P^k$-based OESV scheme.}\label{table:test1}
    		\begin{tabular}{c|c|c|c|c|c|c|c} % 三线表不能有竖线,l-left,c-center,r-right
    			\bottomrule[1.0pt]
    			%三线表-top 线
    			$k$	&	$N_x$ &$L^1$ error& rate & $L^2$ error & rate & $L^\infty$ error &  rate   \\ \hline \multirow{6}*{$1$}

    			&512&    7.23e-05&	-&	8.24e-05&	-&	1.65e-04&-\\ 
    			&1024&  1.65e-05&	2.13&	1.84e-05&	2.16&	3.42e-05&	2.27\\ 
    			&2048&  4.00e-06&	2.04&	4.45e-06&	2.05&	7.77e-06&	2.14\\ 
    			&4096&  9.92e-07&	2.01&	1.10e-06&	2.01&	1.86e-06&	2.06\\ 
    			&8192&  2.47e-07&	2.00&	2.75e-07&	2.00&	4.55e-07&	2.03\\

    			\hline \multirow{6}*{$2$}

    			&256&    9.71e-07&	-&	1.08e-06&	-&	2.64e-06&	-\\ 
    			&512&    7.33e-08&	3.73&	8.32e-08&	3.70&	2.54e-07&	3.38\\ 
    			&1024&  6.59e-09&	3.48&	7.77e-09&	3.42&	2.80e-08&	3.19\\ 
    			&2048&  6.84e-10&	3.27&	8.36e-10&	3.22&	3.31e-09&	3.08\\ 
    			&4096&  7.75e-11&	3.14&	9.74e-11&	3.10&	4.07e-10&	3.02\\ 
    			
    			\hline \multirow{6}*{$3$}		
   
    			&128&    1.14e-07&	-&	1.28e-07&	-&	2.80e-07&	-\\ 
    			&256&    3.64e-09&	4.96&	4.13e-09&	4.95&	1.16e-08&	4.60\\ 
    			&512&    1.20e-10&	4.93&	1.41e-10&	4.87&	5.42e-10&	4.42\\ 
    			&1024&  4.44e-12&	4.75&	5.57e-12&	4.66&	2.85e-11&	4.25\\ 
    			&2048&  2.31e-13&	4.27&	3.14e-13&	4.15&	1.83e-12&	3.96\\ 			
    			
    			\toprule[1.0pt]
    		\end{tabular}
    	\end{center}
    \end{table}

\end{exmp}

\begin{exmp}[pentagram discontinuities]\label{ex:2Dadvec2}
	This example simulates the 2D linear advection equation $u_t+u_x+u_y = 0$ on the spatial domain $\Omega=[0,1]^2$ with periodic boundary conditions and the following discontinuous initial data:
	\begin{equation*}
		u_0(x,y)=
		\begin{cases}
			1 ,\quad& r\leq \frac{1}{8}(3+3^{\sin(5\theta)}),\\
			0,\quad & \text{otherwise},\\
		\end{cases}
		\qquad \theta =
		\begin{cases}
			{\rm arccos}(\frac{x}{r}) ,\quad & y\geq 0,\\
			2\pi-{\rm arccos}(\frac{x}{r}),\quad &y<0,\\
		\end{cases}
	\end{equation*}	
	where $r = \sqrt{x^2+y^2}$. We divide  $\Omega$ into $320\times320$ uniform rectangular cells and conduct the simulation by the OESV schemes up to $t=1.8$.  \cref{fig:2Dadvec2} shows that the numerical solutions by the OESV schemes, which effectively capture the structure of the pentagram-shaped discontinuities.
	
	\begin{figure}[!t]
		\centering

		\includegraphics[width=.32\textwidth]{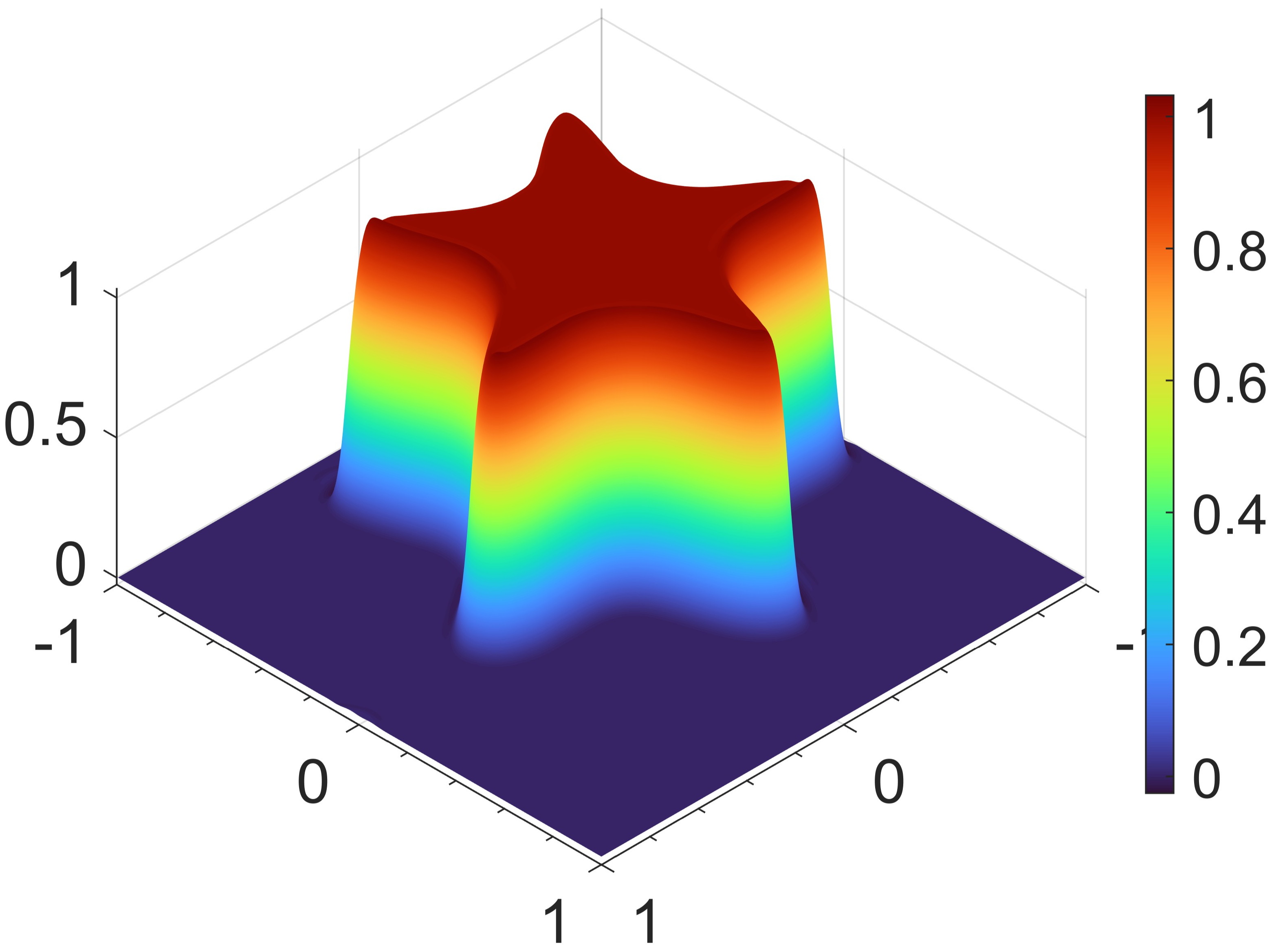}\hfill 
		\includegraphics[width=.32\textwidth]{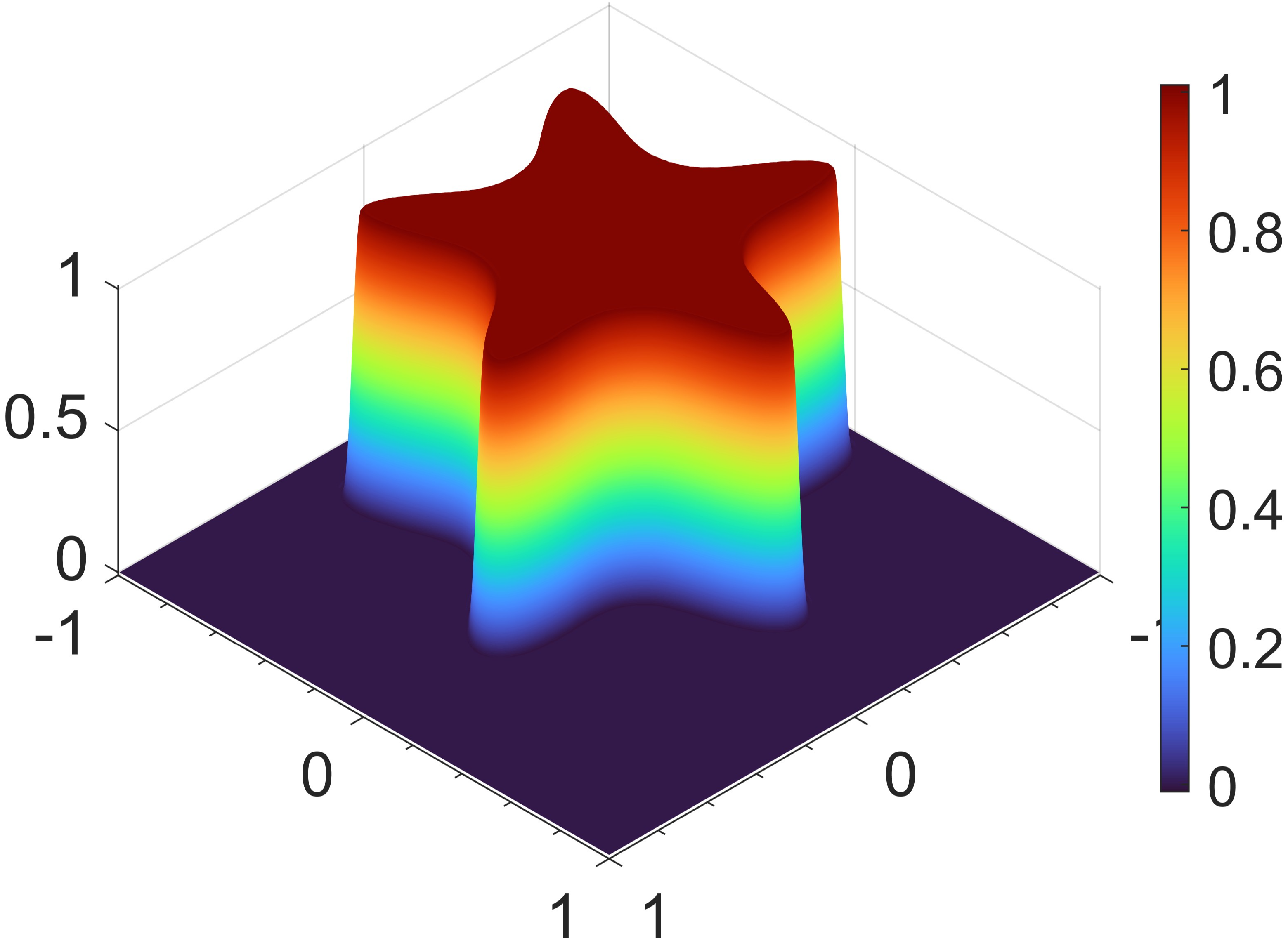}\hfill 
		\includegraphics[width=.32\textwidth]{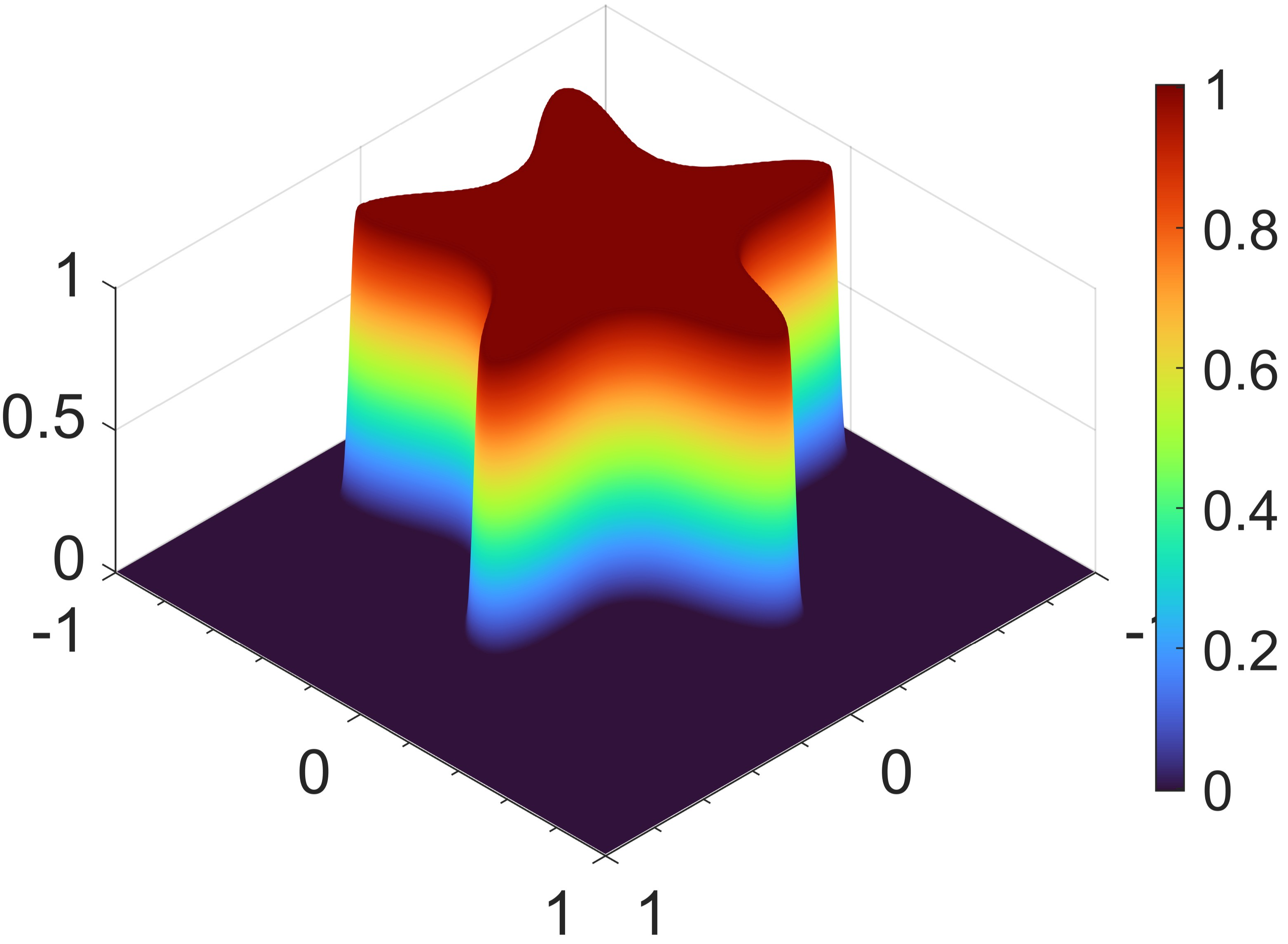}

		\includegraphics[width=.32\textwidth]{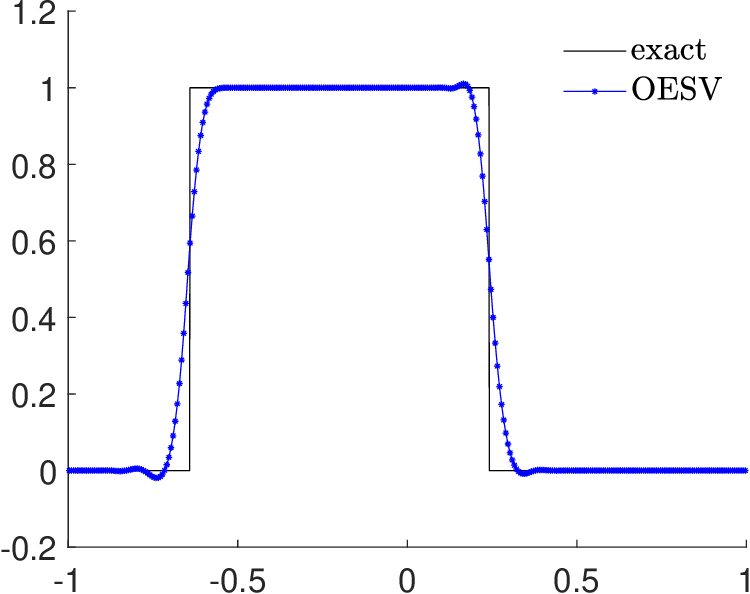}\hfill 
		\includegraphics[width=.32\textwidth]{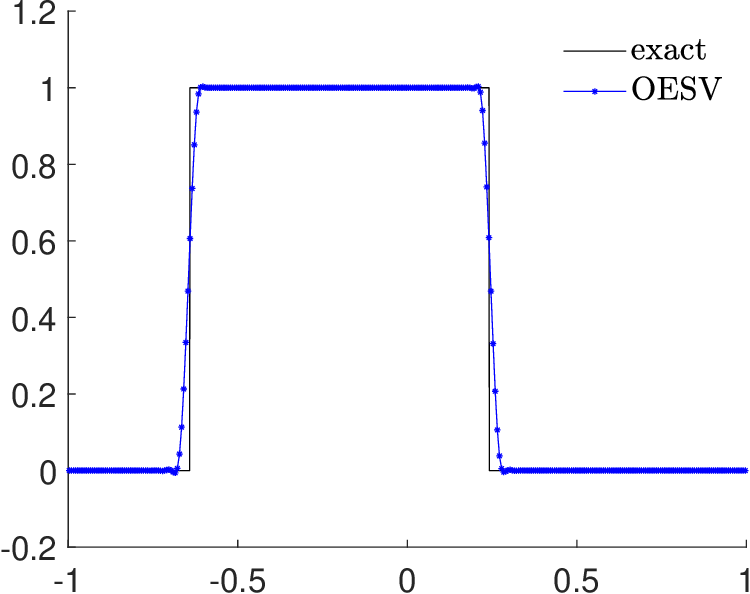}\hfill 
		\includegraphics[width=.32\textwidth]{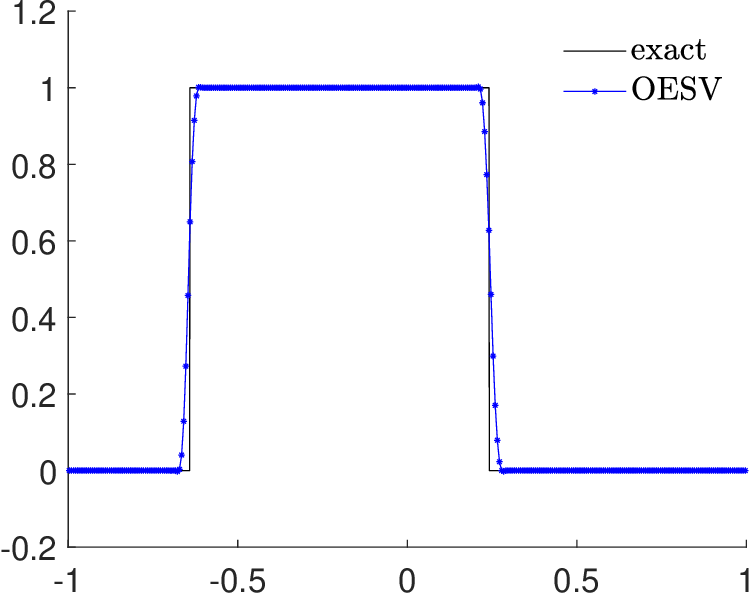} 
		
		\caption{OESV solutions (top) at $t=1.8$ and their cut along $y=-0.25$ (bottom) for \cref{ex:2Dadvec2}. From left to right: $\mathbb Q^1$, $\mathbb Q^2$, and $\mathbb Q^3$ approximations. 
		}
		\label{fig:2Dadvec2}
	\end{figure}

\end{exmp}

\subsection{1D compressible Euler equations}
This subsection presents several examples of the 1D compressible Euler equations ${\bf u}_t + {\bf f}({\bf u})_x = {\bf 0}$ with ${\bf u}=(\rho,\rho v, E)^\top$ and ${\bf f}({\bf u}) = (\rho v, \rho v^2 + p, (E+p)v)^\top$, 
where $\rho$ is the density, $v$ is the velocity, $p$ is the pressure, and $E=\frac{p}{\gamma - 1} + \frac{1}{2}\rho v^{2}$ represents the total energy. The adiabatic index is taken as $\gamma=1.4$ unless otherwise stated.

\begin{exmp}[Riemann problems]\label{ex:1Deuler2}
    This example considers two classical Riemann problems for the 1D Euler equations. 
    The first is the Sod problem with the initial conditions $(\rho_0,v_0,p_0)=(1,~0,~1)$ for $x<0$ and $(0.125,~0,~0.1)$ for $x>0$. 
    The second is the Lax problem with  $(\rho_0,v_0,p_0)= (0.445,~0.698,~3.528)$ for $x<0$ and $(0.5,~0,~0.571)$ for $x>0$. 
    For both cases, we take the domain $\Omega = [-5,5]$ with outflow boundary conditions and conduct the simulation up to $t=1.3$. \cref{fig:1DRMP} presents the numerical solutions to the two Riemann problems computed by the $\mathbb{P}^2$-based OESV scheme with $256$ uniform cells. We observe that the OESV scheme captures the shock and the contact discontinuity effectively and suppresses spurious oscillations. 

    \begin{figure}[!htb]
	\centering
	\begin{subfigure}[h]{.48\linewidth}
		\centering
		\includegraphics[width=.99\textwidth]{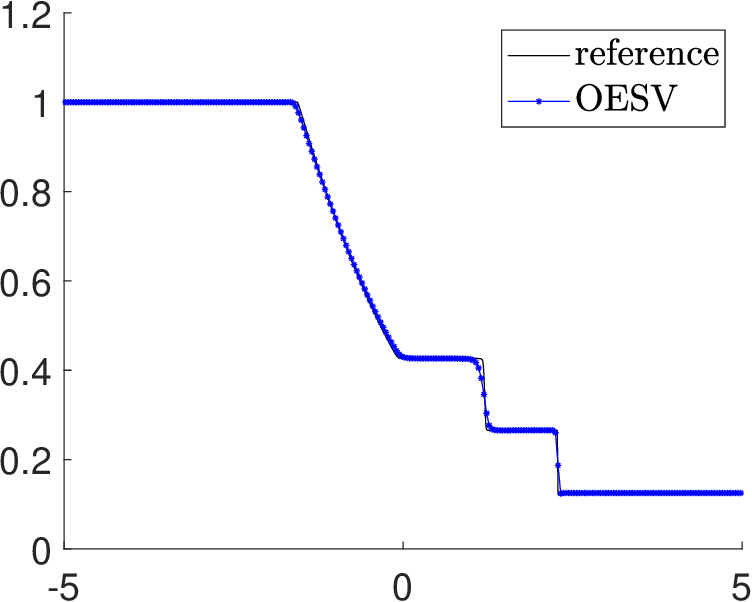}
		\subcaption{Sod's problem}	
	\end{subfigure}	
	\hfill 
	\begin{subfigure}[h]{.48\linewidth}
		\centering
		\includegraphics[width=.99\textwidth]{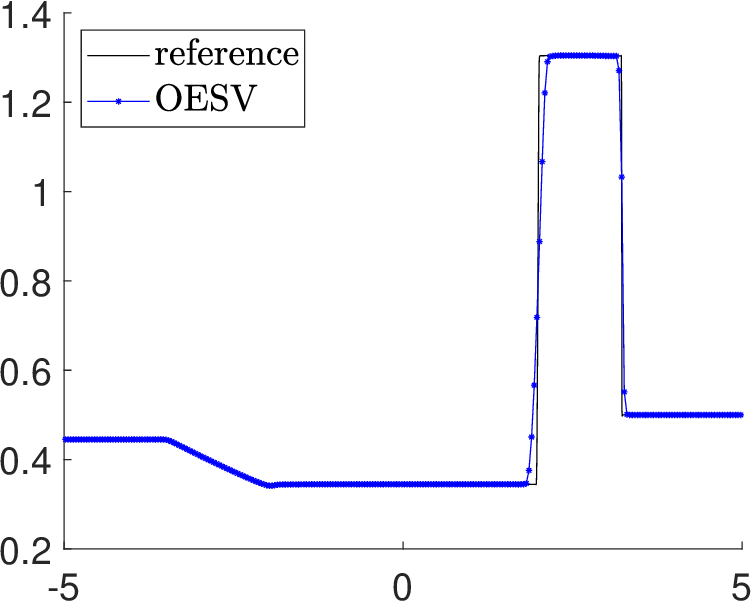}
		\subcaption{Lax's problem}	
	\end{subfigure}
	\caption{Densities of two Riemann problems at $t=1.3$ computed by OESV scheme.}
	\label{fig:1DRMP} % Cell averages are plotted. 
\end{figure}
    
\end{exmp}

\begin{exmp}[Blast problems]\label{ex:1Deuler3}
    This example simulates two blast problems. The first is 
    the interaction of two blast waves proposed by Woodward and Colella. The spatial domain is $\Omega = [0,1]$ with reflective boundary conditions. The initial solution $(\rho_0,v_0,p_0)$ is defined as $(1,0, 10^3)$ for $0<x<0.1$,  as $(1,0, 10^{-2})$ for $0.1<x<0.9$, and as $(1,0, 10^2)$ for $0.9<x<1$. 
    \cref{fig:tb} shows the numerical results at $t=0.038$ computed by the $\mathbb{P}^2$-based OESV scheme on a uniform mesh of $640$ cells. The reference solution is obtained using the $\mathbb{P}^2$-based OEDG scheme \cite{Peng2023OEDGOD} with $10000$ uniform cells. As shown in \cref{fig:tb}, no spurious oscillations are observed near the discontinuities.

In the second test case, we investigate the Sedov blast problem on the domain \(\Omega = [-2, 2]\). This problem models the expanding wave caused by an intense explosion in a perfect gas, involving shocks and extremely low pressure. The initial conditions are \((\rho_0, v_0, E_0) = (1, 0, 10^{-12})\) for all cells except the center one, which is initialized with \((\rho_0, v_0, E_0) = (1, 0, 3200000/h)\), where \(h\) denotes the uniform mesh size. 
We simulate this problem up to \(t = 0.001\) using the \(\mathbb{P}^2\)-based OESV scheme on a uniform mesh with 129 cells. It is worth mentioning that we do not apply any positivity-preserving limiter. The numerical results are presented in \cref{fig:sedov}. The OESV scheme provides a satisfactory simulation without spurious oscillations, indicating its good robustness.

\end{exmp}

\begin{figure}[!htb]
	\centering
	\begin{subfigure}[h]{.48\linewidth}
		\centering
		\includegraphics[width=.99\textwidth]{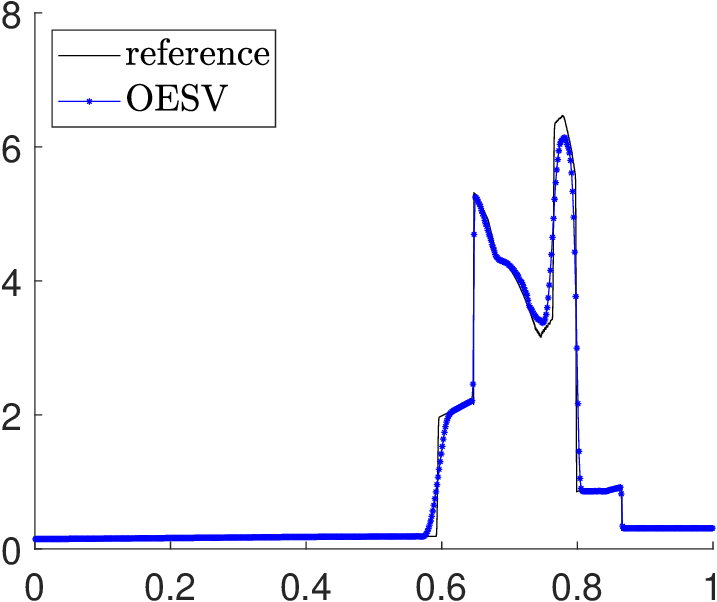}
		\subcaption{Woodward–Colella blast wave,~ $t=0.038$}
            \label{fig:tb}
	\end{subfigure}	
	\hfill 
	\begin{subfigure}[h]{.48\linewidth}
		\centering
		\includegraphics[width=.99\textwidth]{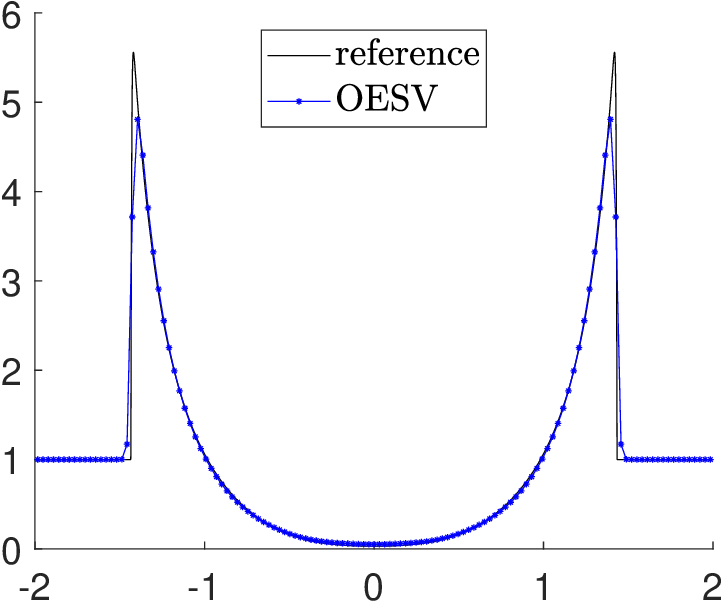}
		\subcaption{Sedov blast,~ $t=0.001$}
            \label{fig:sedov}
	\end{subfigure}
	\caption{Densities of \cref{ex:1Deuler3} computed by the OESV scheme.}
\end{figure}

\subsection{2D compressible Euler equations}
This subsection presents several benchmark test cases of the 2D Euler equations, which can be written in the form of \eqref{eq:HCL} with $ {\bf u} = (\rho, \rho {\bf v}, E)$ and $\quad {\bf f} ({\bf u}) = ( \rho {\bf v},  \rho {\bf v} \otimes {\bf v} + p {\bf I}, (E+p) {\bf v})$. 
Here, $\rho$ is the density, ${\bf v}$ represents the velocity field, $p$ is the pressure, and $E = \frac{p}{\gamma -1} + \frac{1}{2} \rho |{\bf v}|^2$ denotes the total energy. Unless otherwise specified, we set the adiabatic index $\gamma = 1.4$.

\begin{exmp}[double Mach reflection]\label{ex:doublmach}
    The double Mach reflection problem is a classical test case for assessing the capabilities of numerical schemes in handling strong shocks and their interactions. This problem describes a Mach 10 that initially forms an angle of $60^{\circ}$ relative to the bottom boundary of the spatial domain $\Omega = [0,4] \times [0,1]$. The initial conditions are defined as:
    $$(\rho_0,u_0,v_0,p_0) = \begin{cases}
	 	(8,8.25\cos(\frac{\pi}{6}),-8.25\sin(\frac{\pi}{6}),116.5),\quad &x<\frac{1}{6}+\frac{y}{\sqrt{3}},\\
	 	(1.4,0,0,1),\quad &x>\frac{1}{6}+\frac{y}{\sqrt{3}}.
	 \end{cases}$$ 
    The inflow boundary conditions are applied on the left boundary, and the outflow boundary conditions are applied on the right boundary. For the upper boundary, the postshock condition is imposed in the segment from $x=0$ to $x=\frac16+\frac1{\sqrt{3}}(1+20t)$, while the preshock condition is used for the remaining part. For the lower boundary, the postshock condition holds from $x=0$ to $x=1/6$, and the reflective boundary condition is applied to the rest. 
    The numerical solution is computed by the 2D $\mathbb{Q}^2$-based OESV scheme on a uniform rectangular mesh with $h_x=h_y=1/480$. \cref{fig:doublemach} displays the density contours of the numerical solution at $t=0.2$. The result shows that the proposed OESV scheme resolves the flow structure clearly and eliminates nonphysical oscillations effectively.

\begin{figure}[!htb]
	\centering
	\begin{subfigure}[h]{.64\linewidth}
		\centering
		\includegraphics[width=.99\textwidth]{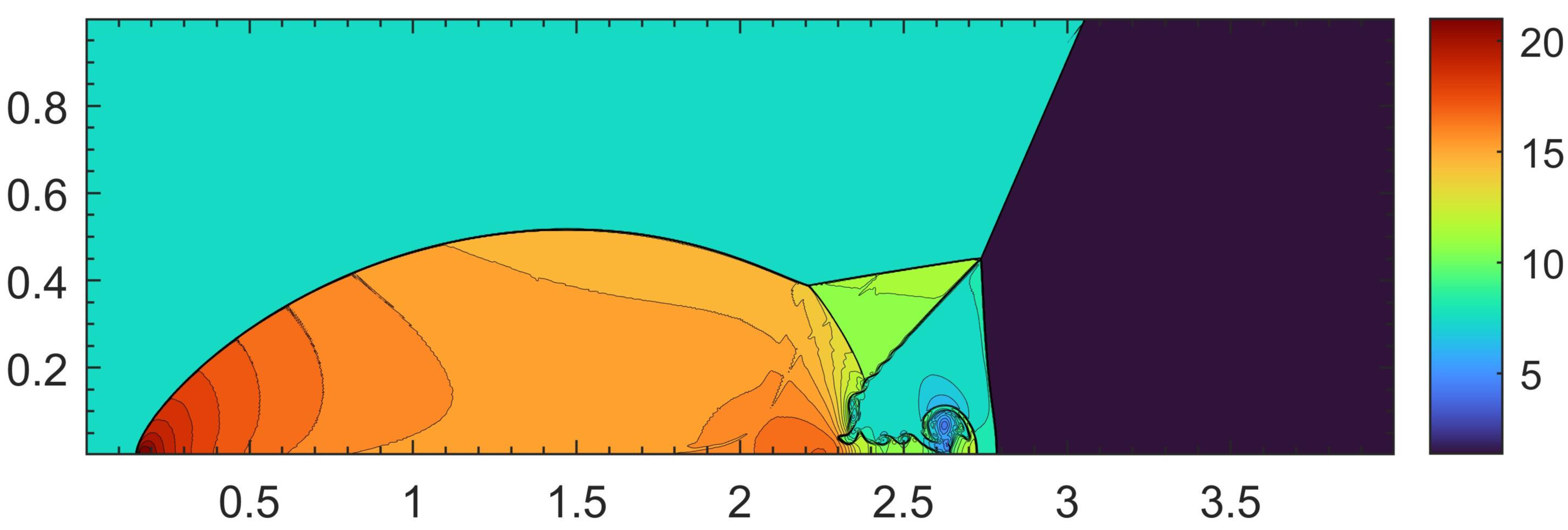}
	\end{subfigure}	
	\hfill 
	\begin{subfigure}[h]{.35\linewidth}
		\centering
		\includegraphics[width=.99\textwidth]{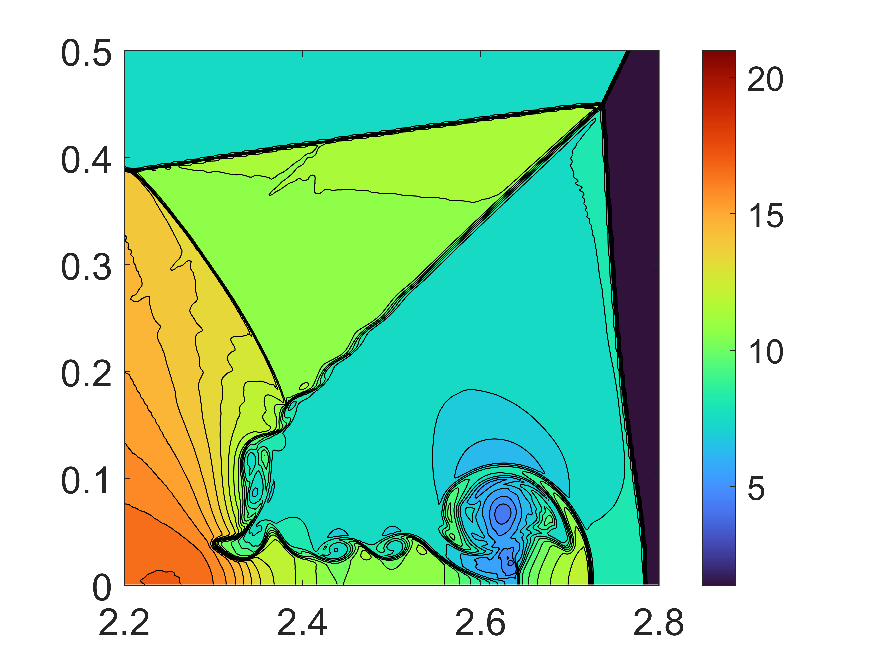}
	\end{subfigure}
	\caption{The contour plot of density (left) and its close-up (right) at $t=0.2$ obtained by the $\mathbb{Q}^2$-based OESV scheme with $h_x=h_y=1/480$. }
         \label{fig:doublemach}
  \end{figure}

\end{exmp}

\begin{exmp}[Mach 2000 jet]\label{ex:jet2000}
    This example simulates a challenging jet problem in the spatial domain $\Omega = [0,1]\times[-0.25,0.25]$ to demonstrate the effectiveness of the proposed OESV scheme. The ratio of heat capacity is set to $\gamma = \frac{5}{3}$. Initially, $\Omega$ is full of a stationary fluid characterized by $(\rho,{\bf v},p) = (0.5,0,0,0.4127)$. A Mach-$2000$ jet with the state $(\rho,{\bf v},p) = (5,800,0,0.4127)$ is injected into $\Omega$ from the left boundary within the range $y=-0.05$ to $0.05$. The remaining boundaries are imposed with the outflow boundary conditions. 
    We simulate the jet using the $\mathbb{Q}^2$-based OESV scheme with $320 \times 160$ uniform cells. 
    The numerical results at $t = 0.001$ are displayed in \cref{fig:jet2000}. We observe that the proposed OESV scheme successfully captures the intricate structures of the jet flow, including the bow shock and the shear layer, and does not procedure any obvious spurious oscillations.
%We also impose a positivity-preserving limiter in our simulations. 

    \begin{figure}[!htb]
		\centering
		\begin{subfigure}[h]{0.49\linewidth}
		\centering
		\includegraphics[width=1\textwidth]{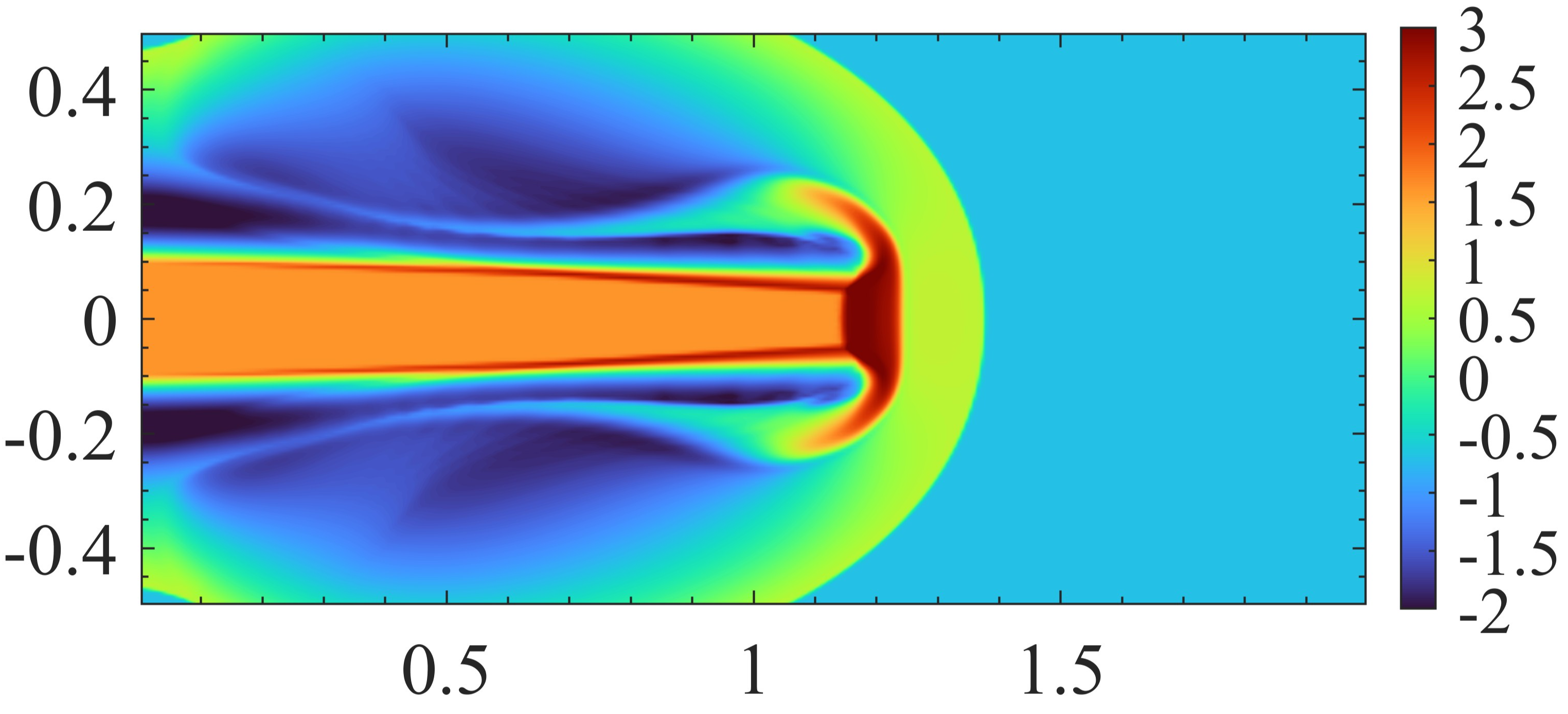}
		\subcaption{\small Density logarithm}
		\vspace{3mm}
		\end{subfigure}
	\hfill 
		\begin{subfigure}[h]{0.49\linewidth}
			\centering
			\includegraphics[width=1\textwidth]{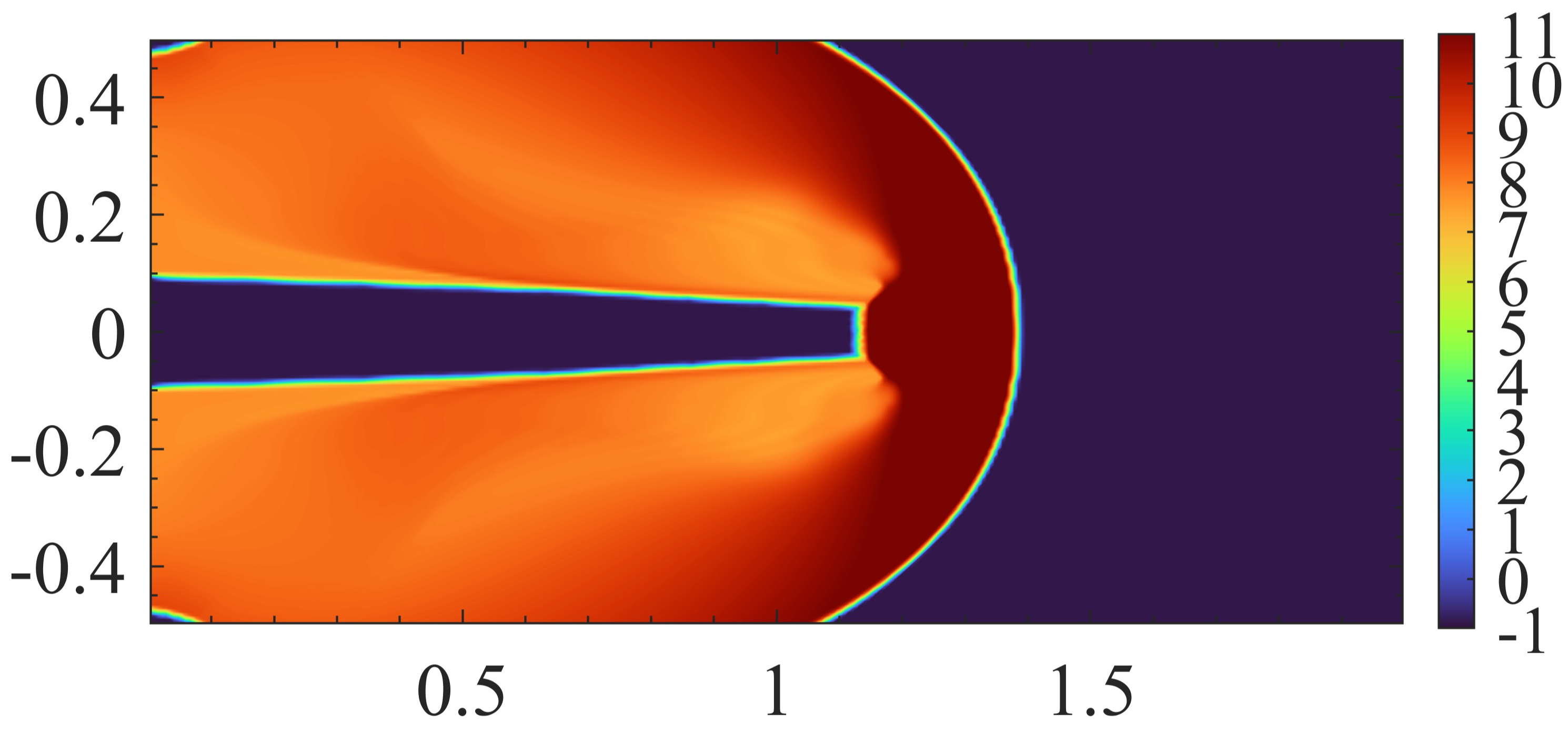}
			\subcaption{\small Pressure logarithm}
			\vspace{3mm}
		\end{subfigure}
	
			\begin{subfigure}[h]{0.49\linewidth}
		\centering
		\includegraphics[width=1\textwidth]{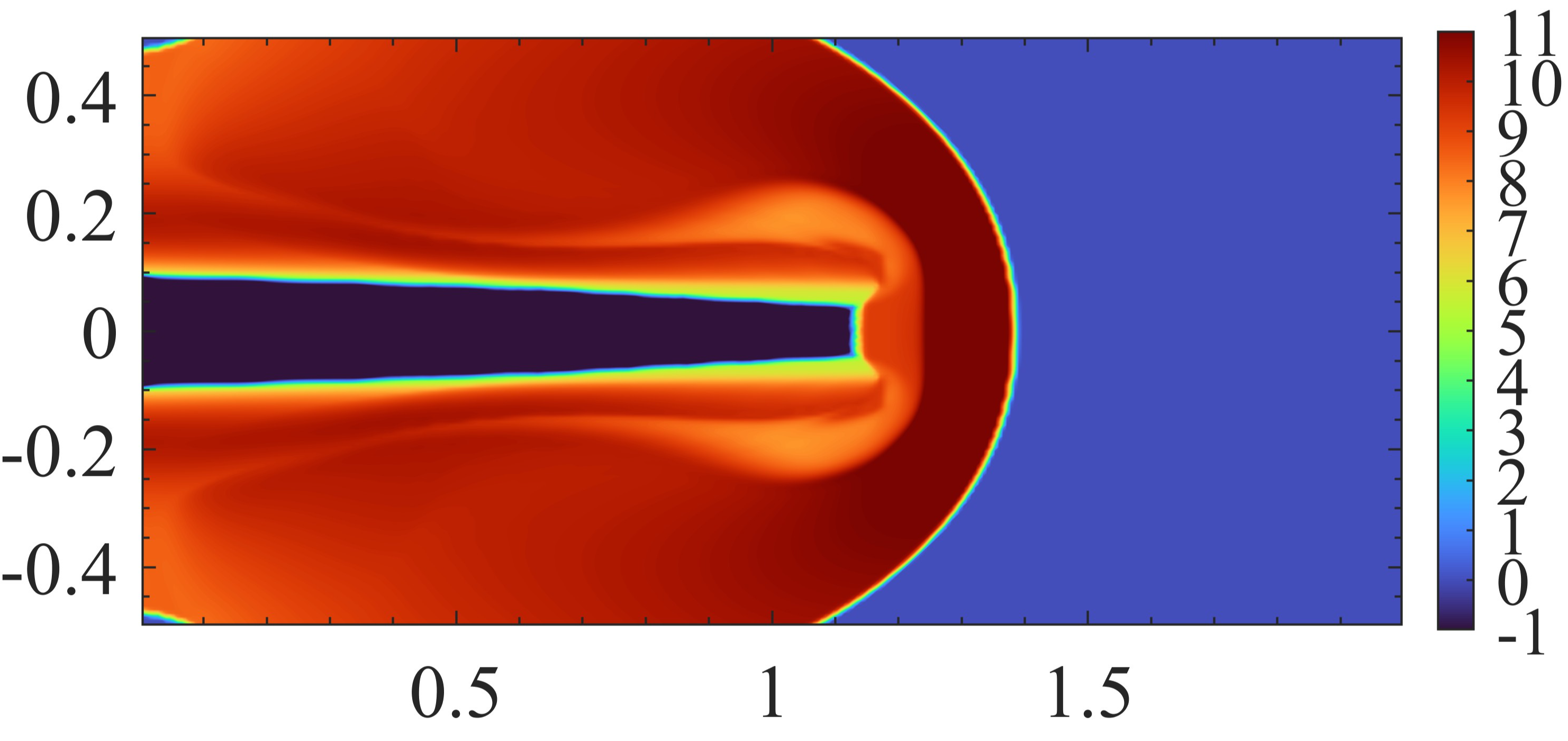}
		\subcaption{\small Temperature logarithm}
	\end{subfigure}
\hfill 
 			\begin{subfigure}[h]{0.49\linewidth}
 	\centering
 	\includegraphics[width=1\textwidth]{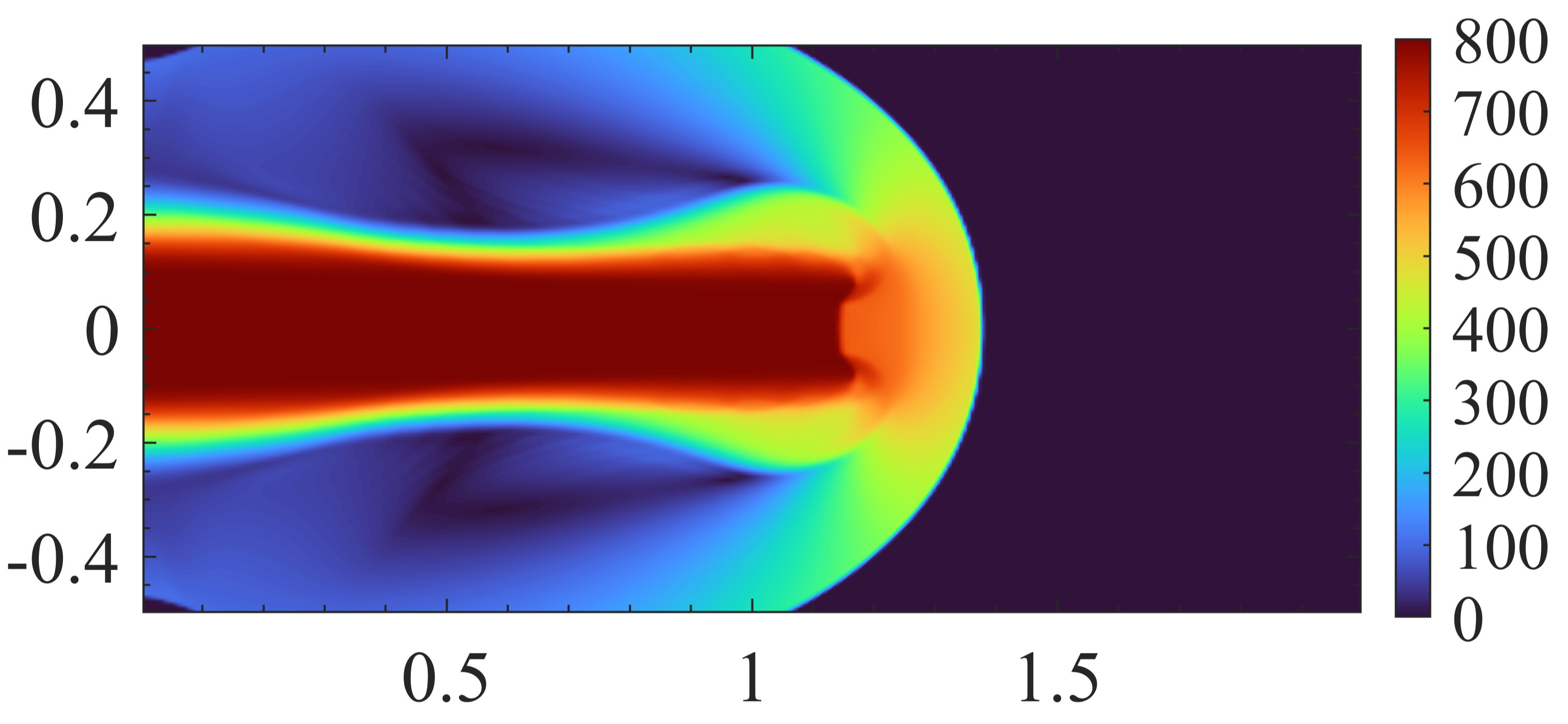}
 	\subcaption{\small Velocity magnitude}
 \end{subfigure}
		\caption{Numerical results at $t=0.001$ for the Mach 2000 jet problem.}
		\label{fig:jet2000}
	\end{figure}
\end{exmp}

Additional numerical examples can be found in \Cref{app:numex}.

\section{Conclusions}

In this paper, we have established a novel connection between the spectral volume (SV) method and the discontinuous Galerkin (DG) method for solving hyperbolic conservation laws, inspired by the Galerkin form of the SV method \cite{Cao2021AnalysisOS}. By demonstrating that the SV method can be represented in a DG form with a distinct inner product under specific subdivision assumptions, we have provided a unifying perspective for these two widely-used numerical methodologies. 
This insight allowed us to successfully extend the oscillation-eliminating (OE) technique, proposed by Peng, Sun, and Wu \cite{Peng2023OEDGOD}, to develop a new fully-discrete oscillation-eliminating SV (OESV) method. The OE procedure, being non-intrusive, efficient, and straightforward to implement, serves as a robust post-processing filter that effectively suppresses spurious oscillations. 
Our comprehensive framework, grounded in a DG perspective, facilitates rigorous theoretical analysis of the stability and accuracy of both general Runge--Kutta SV (RKSV) schemes and the novel OESV method. Specifically, for the linear advection equation, we identified a crucial upwind condition for stability and established optimal error estimates for the OESV schemes. The challenges arising from the nonlinearity of the OESV method were addressed through error decomposition and by treating the OE procedure as additional source terms in the RKSV schemes. 
Extensive numerical experiments have validated our theoretical analysis, demonstrating the effectiveness and robustness of the proposed OESV method across a range of benchmark problems. 
In conclusion, this work not only enhances the theoretical understanding of SV schemes but also significantly improves their practical application. It opens opportunities for future research to explore further extensions and applications of the OESV method for high-resolution simulations of various complex hyperbolic systems.

\section{Appendix}
This appendix provides some additional numerical examples and the proofs of several 
auxiliary propositions.

\subsection{Additional numerical examples}\label{app:numex}

\begin{exmp}[smooth problem]\label{ex:1Deuler1}
	This example examines a smooth problem of the 1D compressible Euler equations on the spatial domain $\Omega = [0,2\pi]$ with the periodic boundary conditions. The exact solution is given by $\rho(x,t) = 2 + 2\sin^2(x-t)$, 
	$v(x,t) = 1$, and $p(x,t) = 2$. 
	The problem is simulated up to $t=1.1$ using the $\mathbb{P}^k$-based OESV scheme with a CFL number $C_{\mathrm{CFL}}=\frac{0.95}{2k+1}$.
	The errors and corresponding convergence rates are listed in \cref{table:test3}. It is observed that the $\mathbb{P}^k$-based OESV scheme achieves the optimal convergence rate of $(k+1)$th-order for this  problem. 
	
	\begin{table}[!htb]
		\centering
		% h-here,t-top,b-bottom,优先级依次下降
		\caption{Errors and convergence rates for $\mathbb P^k$-based OESV scheme for 1D Euler equations.}\label{table:test3} 
		\begin{center}
			% 居中
			
			\begin{tabular}{c|c|c|c|c|c|c|c} % 三线表不能有竖线,l-left,c-center,r-right
				\bottomrule[1.0pt]
				%三线表-top 线
				$k$	&	$N_x$ &  $L^1$ error & rate & $L^2$ error & rate & $L^\infty$ error & rate   \\
				\hline
				%三线表-middle 线
				\multirow{7}*{$1$}
				
				&256&        1.62e-03&	       -&	7.70e-04&	       -&	6.24e-04&	       -\\ 
				&512&        2.92e-04&	2.47&	1.38e-04&	2.48&	1.14e-04&	2.46\\ 
				&1024&      6.65e-05&	2.13&	2.99e-05&	2.21&	2.35e-05&	2.27\\ 
				&2048&      1.60e-05&	2.05&	7.15e-06&	2.07&	5.44e-06&	2.11\\ 
				&4096&      3.95e-06&	2.02&	1.77e-06&	2.02&	1.34e-06&	2.02\\ 
				&8192&      9.83e-07&	2.01&	4.40e-07&	2.00&	3.34e-07&	2.00\\ 
				&16384&    2.45e-07&  	2.00&	1.10e-07&	2.00&	8.36e-08&	2.00\\ 
				&32768&    6.12e-08&	2.00&	2.75e-08&	2.00&	2.09e-08&	2.00\\

				\hline
				%		&&&&&&&&&\\
				\multirow{4}*{$2$}
				
				&256&      8.44e-06&	  -&	  4.28e-06&	-&	4.24e-06&	-\\ 
				&512&      7.68e-07&	  3.46&	  4.03e-07&	3.41&	4.17e-07&	3.35\\ 
				&1024&    9.00e-08&	  3.09&	  4.55e-08&	3.15&	4.74e-08&	3.14\\ 
				&2048&    1.10e-08&	  3.03&	  5.49e-09&	3.05&	5.71e-09&	3.05\\ 
				&4096&    1.37e-09&	  3.01&	  6.77e-10&	3.02&	7.04e-10&	3.02\\ 
				&8192&    1.70e-10&	  3.01&	  8.41e-11&	3.01&	8.81e-11&	3.00\\

				\hline
				
				\multirow{4}*{$3$}
				
				&256&    1.77e-08&	-&	8.17e-09&	-&	7.20e-09&	-\\ 
				&512&    6.32e-10&	4.81&	2.99e-10&	4.77&	3.00e-10&	4.59\\ 
				&1024&  2.89e-11&	4.45&	1.36e-11&	4.46&	1.41e-11&	4.41\\ 
				&2048&  1.53e-12&	4.24&	7.50e-13&	4.18&	8.70e-13&	4.02\\ 
				%	&&&&&&&&&\\
				\toprule[1.0pt]
				%三线表-底线
			\end{tabular}
		\end{center}
	\end{table}
\end{exmp}

\begin{exmp}[smooth problem]\label{ex:2Dadvec1}
	This example serves as an accuracy test for the 2D $\mathbb{Q}^k$-based OESV scheme for the linear advection equation:
	\begin{equation}\label{eq:2Dadvec}
		u_t+u_x+u_y = 0,\quad (x,y) \in \Omega=[0,1]\times[0,1],
	\end{equation}
	with the periodic boundary conditions. 
	The simulation is initialized with $u_0(x,y) =\sin^2\left( \pi(x+y) \right) $. We list the numerical errors on a mesh of $N_x \times N_y$ cells at $t=1.1$ and the corresponding convergence rates in \cref{table:2Dadvec1}. 
	On coarse meshes, it is observed that the high-order damping effect dominates the numerical errors, leading to convergence rates exceeding $k + 1$. This effect was also observed in the OEDG method \cite{Peng2023OEDGOD}.

	\begin{table}[!htb]
		\centering
		% h-here,t-top,b-bottom,优先级依次下降
		\begin{center}
			% 居中 
			\caption{Errors and convergence rates of 2D $\mathbb Q^k$-based OESV scheme for \Cref{ex:2Dadvec1}.}\label{table:2Dadvec1}
			\begin{tabular}{c|c|c|c|c|c|c|c} % 三线表不能有竖线,l-left,c-center,r-right
				\bottomrule[1.0pt]
				%三线表-top 线
				$k$	&	$N_x \times N_y$ &$L^1$ error& rate & $L^2$ error & rate & $L^\infty$ error &  rate   \\ \hline \multirow{6}*{$1$}
				
				&$80\times64$&2.18e-02&-&2.58e-02&-&4.32e-02&-\\
				&$160\times128$&3.68e-03&2.57&4.32e-03&2.58&7.07e-03&2.61\\
				&$320\times256$&6.19e-04&2.57&7.23e-04&2.58&1.23e-03&2.52\\
				&$640\times512$&1.16e-04&2.41&1.31e-04&2.47&2.37e-04&2.38\\
				&$1280\times1024$&2.54e-05&2.20&2.78e-05&2.24&4.80e-05&2.30\\
				&$2560\times2048$&5.97e-06&2.09&6.54e-06&2.09&1.05e-05&2.20\\

				\hline \multirow{6}*{$2$}
				
				&$80\times64$&6.49e-04&-&7.43e-04&-&1.16e-03&-\\
				&$160\times128$&2.31e-05&4.82&2.57e-05&4.85&3.74e-05&4.95\\
				&$320\times256$&1.23e-06&4.23&1.35e-06&4.25&1.92e-06&4.28\\
				&$640\times512$&7.81e-08&3.98&8.46e-08&3.99&1.18e-07&4.03\\
				&$1280\times1024$&5.66e-09&3.79&6.07e-09&3.80&8.20e-09&3.85\\
				&$2560\times2048$&4.68e-10&3.60&5.00e-10&3.60&6.51e-10&3.65\\
				
				\hline \multirow{6}*{$3$}		
				
				&$80\times64$&4.82e-06&-&5.64e-06&-&9.79e-06&-\\
				&$160\times128$&1.65e-07&4.87&1.86e-07&4.92&3.16e-07&4.95\\
				&$320\times256$&5.26e-09&4.97&5.89e-09&4.98&9.48e-09&5.06\\
				&$640\times512$&1.67e-10&4.98&1.86e-10&4.99&2.89e-10&5.04\\
				&$1280\times1024$&5.24e-12&4.99&5.83e-12&4.99&8.48e-12&5.09\\		
				
				\toprule[1.0pt]
			\end{tabular}
		\end{center}
	\end{table}

\end{exmp}

\begin{exmp}[2D smooth problem]\label{ex:2Deuler1}
	The exact solution of this example is smooth, describing a sine wave periodically propagating in the spatial domain $\Omega = [0,2]\times[0,2]$, and given by  $(\rho,{\bf v},p) = (1+0.2\sin(\pi(x+y-t)),0.7,0.3,1)$. 
	We compute the numerical solution using the 2D $\mathbb{Q}^k$-based OESV scheme. The numerical errors and the corresponding convergence rates for the density at $t=2$ are presented in \cref{table:2Deuler1}. We observe the expected convergence rate of $k+1$ for the $\mathbb{Q}^k$-based OESV scheme.
	
	\begin{table}[!htb] %{ Rates are too high, how about k+1 th order RK? Add P3 results}.
		\caption{Errors and convergence rates for 2D $\mathbb Q^k$-based OEDG method for \Cref{ex:2Deuler1}.}\label{table:2Deuler1}
		
		% h-here,t-top,b-bottom,优先级依次下降
		\begin{center}
			\centering
			% 居中
			\begin{tabular}{c|c|c|c|c|c|c|c} % 三线表不能有竖线,l-left,c-center,r-right
				\bottomrule[1.0pt]
				%三线表-top 线
				&	$N_x\times N_y$ &$L^1$ error& rate & $L^2$ error & rate & $L^\infty$ error & rate   \\ \hline \multirow{5}*{$1$}
				
				&80$\times$80&2.14e-04&-&2.38e-04&-&3.38e-04&-\\ 
				&160$\times$160&5.36e-05&2.00&5.95e-05&2.00&8.42e-05&2.00\\ 
				&320$\times$320&1.34e-05&2.00&1.49e-05&2.00&2.10e-05&2.00\\ 
				&640$\times$640&3.35e-06&2.00&3.72e-06&2.00&5.26e-06&2.00\\
				&1280$\times$1280&8.37e-07&2.00&9.29e-07&2.00&1.31e-06&2.00\\ 
				
				\hline \multirow{5}*{$2$}
				
				&80$\times$80&5.71e-08&-&6.90e-08&-&1.55e-07&-\\ 
				&160$\times$160&6.06e-09&3.24&8.33e-09&3.05&2.05e-08&2.92\\ 
				&320$\times$320&2.24e-10&4.76&2.65e-10&4.97&5.59e-10&5.20\\ 
				&640$\times$640&2.04e-11&3.46&2.58e-11&3.36&6.01e-11&3.22\\ 
				&1280$\times$1280&2.60e-12&2.97&3.31e-12&2.96&8.47e-12&2.83\\

				\hline \multirow{5}*{$3$}
				
				&80$\times$80&8.93e-10&-&1.14e-09&-&2.75e-09&-\\ 
				&160$\times$160&4.66e-11&4.26&6.03e-11&4.24&1.47e-10&4.23\\ 
				&320$\times$320&1.77e-12&4.72&2.29e-12&4.72&6.59e-12&4.48\\ 
				&640$\times$640&9.08e-14&4.28&1.07e-13&4.42&2.22e-13&4.89\\

				\toprule[1.0pt]
			\end{tabular}
		\end{center}
	\end{table}
\end{exmp}

\begin{exmp}[shock reflection problem]\label{ex:shockref}
	This example simulates the shock reflection problem \cite{ZHU201780} for the 2D compressible Euler equations in the spatial domain $\Omega = [0,4]\times[0,1]$. The initial conditions for this problem are defined as 
	$
	(\rho_0,{\bf v}_0,p_0) = (1,2.9,0,\frac{5}{7}),
	$, which are also applied as the inflow boundary conditions on the left boundary of $\Omega$. On the upper boundary of $\Omega$, a different set of inflow boundary conditions is applied:
	\begin{equation*}
		(\rho,{\bf v},p)=(1.69997, 2.61934, -0.50632, 1.52819).
	\end{equation*}
	Reflective wall boundary conditions are applied on the lower boundary of $\Omega$, while the right boundary is subjected to outflow boundary conditions.
	
	The numerical solution for this example is obtained using the 2D $\mathbb{Q}^2$-based OESV scheme on a uniform rectangular mesh of $200\times50$ cells. Following \cite{Liu2022AnEO}, we compute the average residue to study the convergence behavior of the numerical solution. The average residue is defined as:
	\begin{equation}\label{eq:avres}
		{\rm Res} := \frac1{4N_xN_y} \sum_{i=1}^{N_x}\sum_{j=1}^{N_y}\sum_{q=1}^4\left|R^{(q)}_{ij} \right|,
	\end{equation}
	where $R^{(q)}_{ij}:= \frac{1}{\tau} ( u_h^{n+1,(q)} -u_h^{n,(q)}  ) $ is the local residue on the cell $[x_{i-\frac12},x_{i+\frac12}]\times [y_{j-\frac12},y_{j+\frac12}]$, and $u_h^{n,(q)}$ represents the $q$th component of the numerical solution at the $n$th time step. 
	
	We plot the logarithm of the average residue over time in \cref{fig:shockref1}. The plot shows that the average residue decreases to the level of machine error in double precision after about $t = 9$, indicating that the numerical solution converges to a steady state as time progresses. The density contour of the numerical solution at $t = 20$ is displayed in \cref{fig:shockref2}, where the shock waves are clearly captured without spurious oscillations.
	
	\begin{figure}[!htb]
		\centering
		\begin{subfigure}[h]{.32\linewidth}
			\centering
			\includegraphics[width=.99\textwidth]{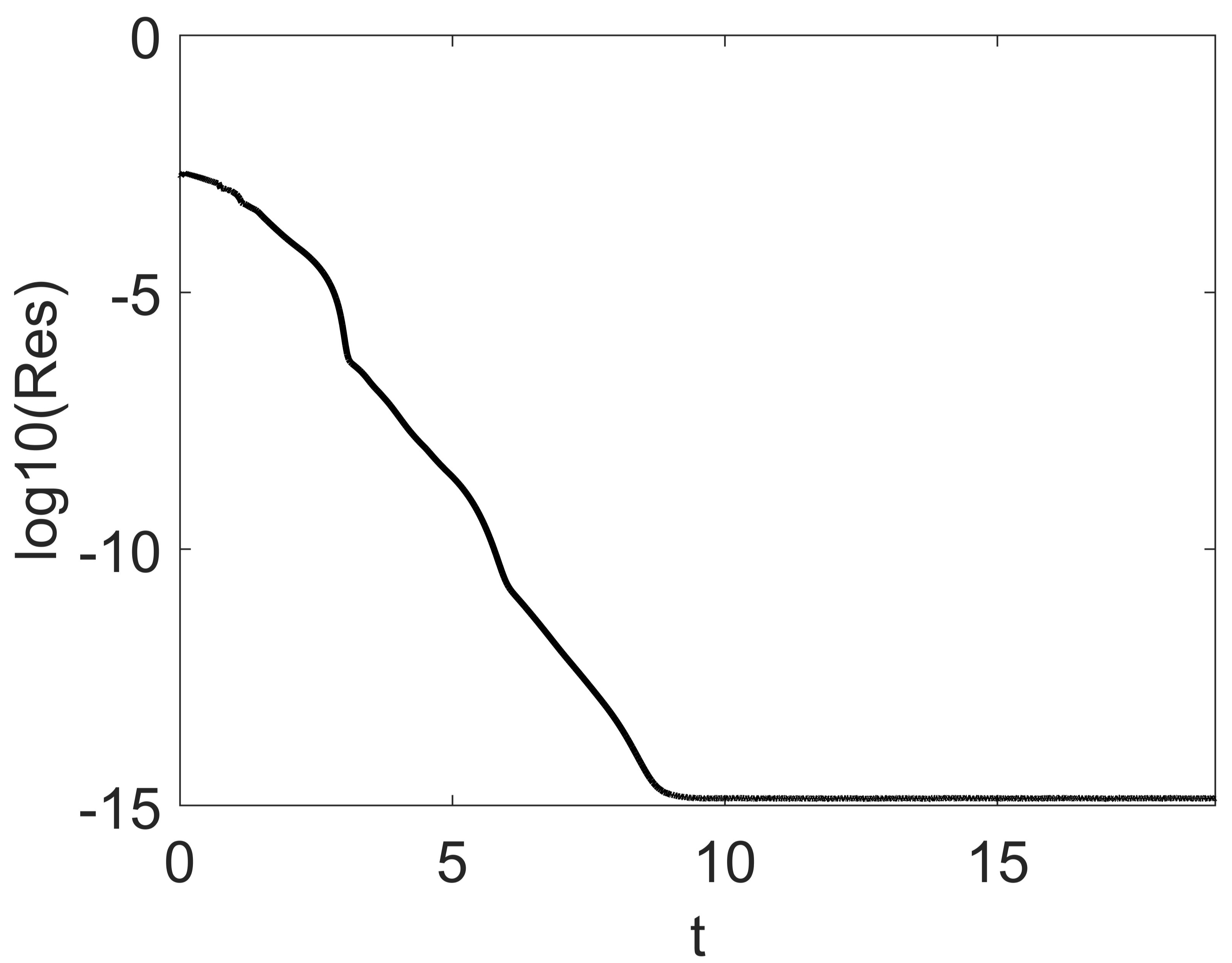}
			\subcaption{Average residue}
			\label{fig:shockref1}
		\end{subfigure}	
		\hfill 
		\begin{subfigure}[h]{.67\linewidth}
			\centering
			\includegraphics[width=.99\textwidth]{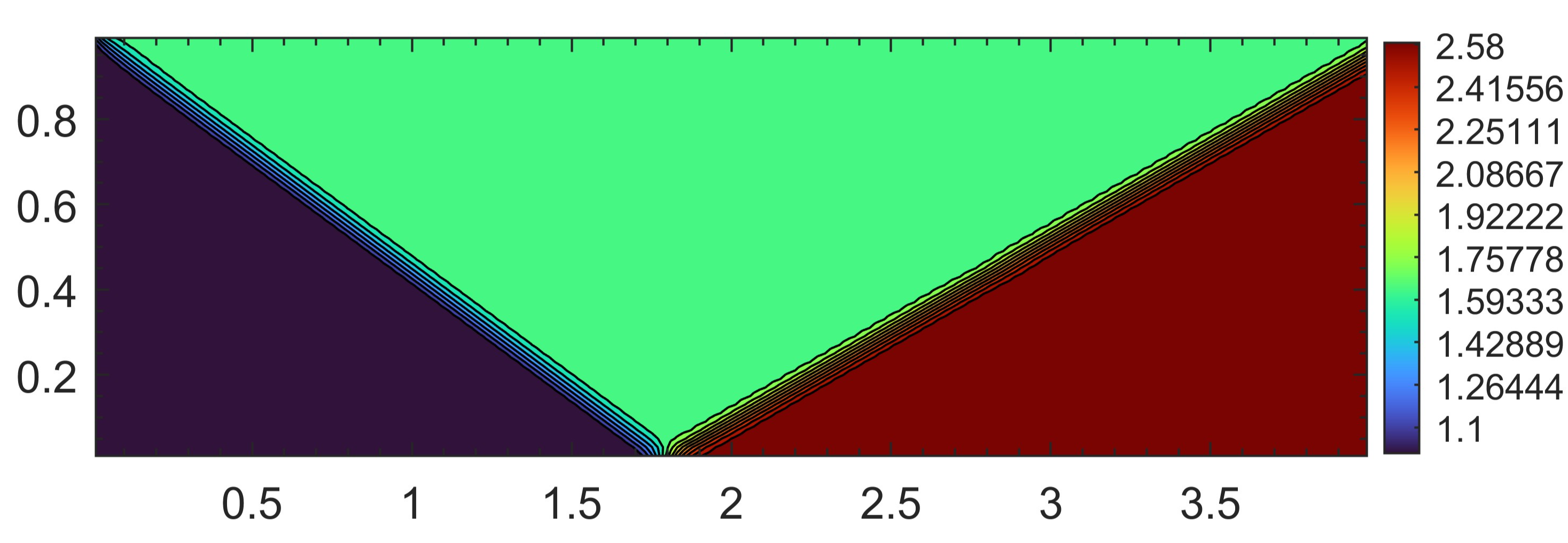}
			\subcaption{Density contour at $t=20$}
			\label{fig:shockref2}
		\end{subfigure}
		\caption{Shock reflection problem simulated by $\mathbb{Q}^2$-based OESV scheme. }
	\end{figure}

\end{exmp}

\begin{exmp}[supersonic flow past two plates]\label{ex:supersonic}
	This example \cite{Zhu2018NumericalSO} simulates a supersonic flow in the spatial domain $\Omega = [0,10] \times [-5,5]$ with the initial conditions
	$(\rho_0, {\bf v}_0,p_0) = \left(1,\cos(\frac{\pi}{12}),\sin(\frac{\pi}{12}),\frac{1}{\gamma M_{\infty}^2} \right),$ where $M_{\infty}=3$ is the Mach number of the free stream. 
	The supersonic flow pasts two plates with an attack angle of $15^{\circ}$. The two plates are placed at $y=\pm2$ with $x\in(2,3)$, and slip boundary conditions are applied on both plates. The inflow boundary conditions are applied on the left and lower boundaries of $\Omega$. The outflow boundary conditions are applied on the upper and right boundaries.
	The simulation is performed using the 2D $\mathbb{Q}^2$-based OESV scheme on a mesh of $200 \times 200$ uniform cells. Similar to \cref{ex:shockref} and \cite{Peng2023OEDGOD}, we display the average residue history and plot the density contour of the numerical solution at $t=100$ in \cref{fig:supersonic}. One can see from \cref{fig:supersonic1} that the numerical solution reaches a steady state at $t=100$. We also observe from \cref{fig:supersonic2} that the OESV scheme captures the flow structure correctly with no nonphysical oscillations.
	
	\begin{figure}[!htb]
		\centering
		\begin{subfigure}[h]{.48\linewidth}
			\centering
			\includegraphics[width=.99\textwidth,height=0.84\textwidth]{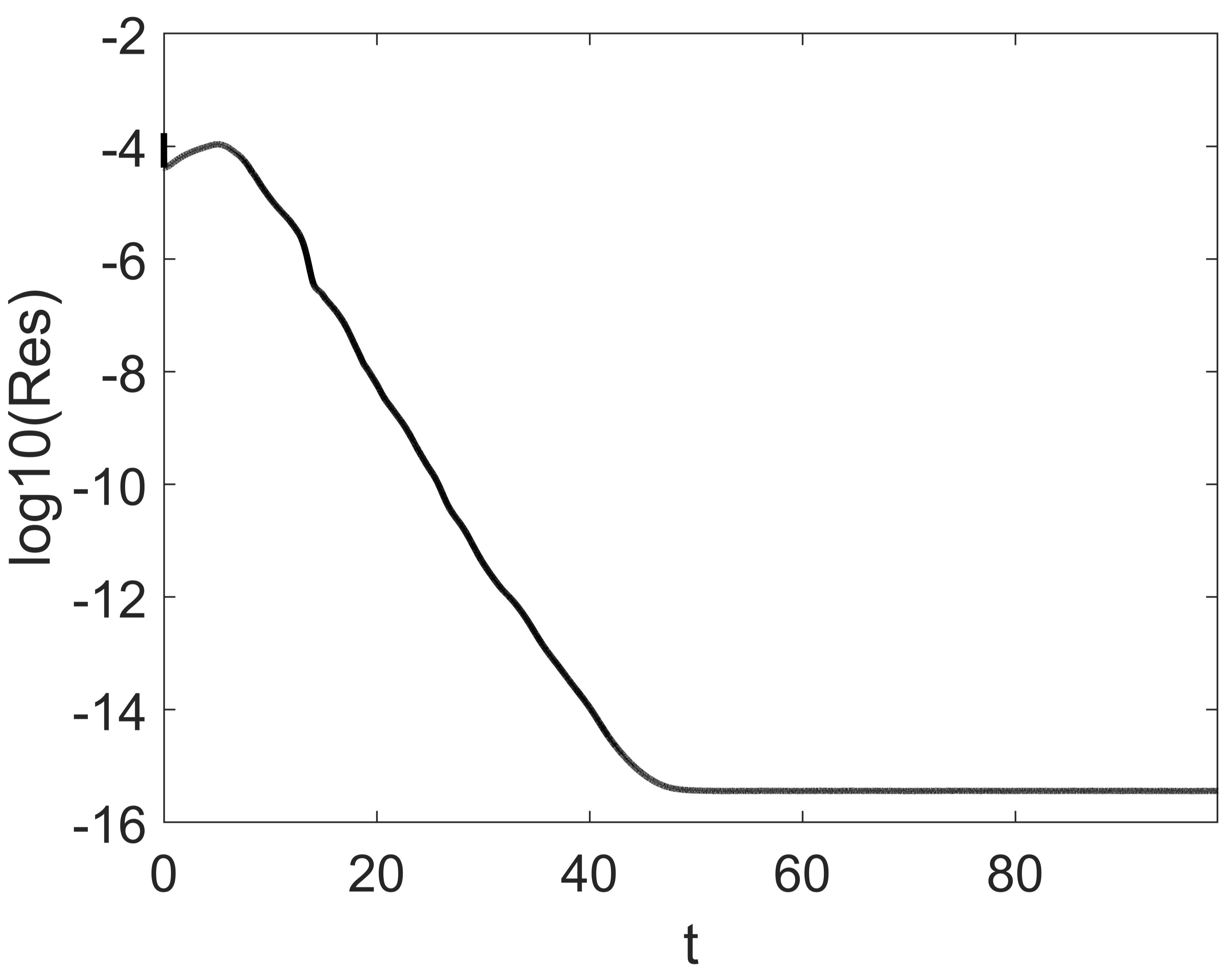}
			\subcaption{Average residue}
			\label{fig:supersonic1}
		\end{subfigure}	
		\hfill 
		\begin{subfigure}[h]{.48\linewidth}
			\centering
			\includegraphics[width=.99\textwidth]{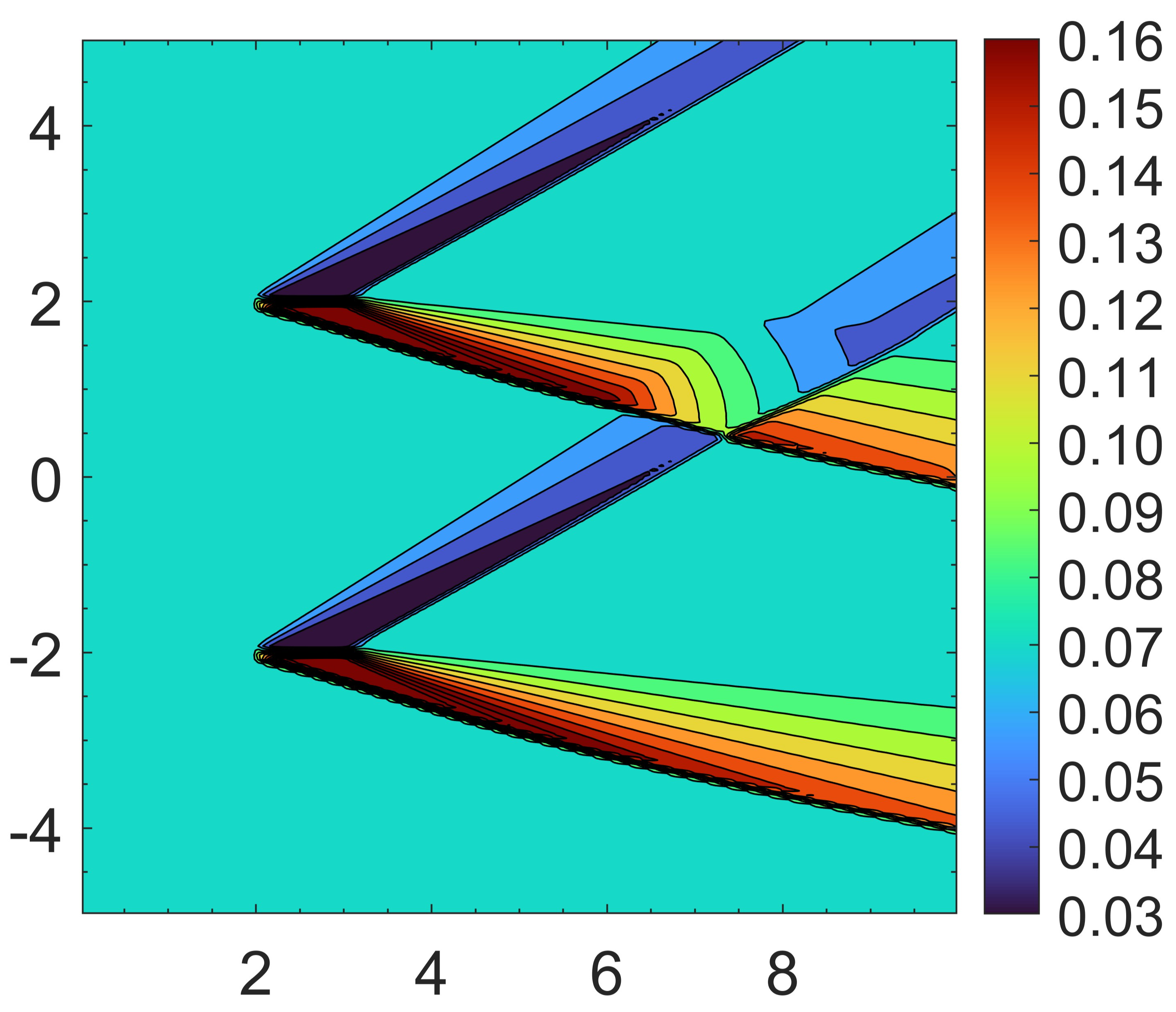}
			\subcaption{Density contour at $t=100$}
			\label{fig:supersonic2}
		\end{subfigure}
		\caption{Supersonic flow problem simulated by $\mathbb{Q}^2$-based OESV scheme. }
		\label{fig:supersonic}
	\end{figure}
	
\end{exmp}

\begin{exmp}[shock-vortex interaction]\label{ex:vortex}
	We consider the interaction between a vortex and a Mach 1.1 shock in the spatial domain $\Omega = [0,2]\times[0,1]$ by solving the 2D compressible Euler equations. The shock is perpendicular to the $x$-axis and is positioned at $x=0.5$, with the left upstream state $(\rho, {\bf v}, p) = (1,1.1\sqrt{\gamma},0,1)$. At $t=0$, an isentropic vortex centered at $(x_c,y_c) = (0.25,0.5)$ is added to the mean flow. The velocity, temperature, and entropy perturbations due to the vortex are defined by 
	\begin{align*}
		\delta {\bf v} = \frac{\varepsilon}{r_c} e^{\alpha(1-\eta^2)}(\bar{y}, -\bar{x}), \quad
		\delta T = -\frac{(\gamma-1)\varepsilon^2}{4\alpha\gamma} e^{2\alpha(1-\eta^2)}, \quad
		\delta S =0,
	\end{align*}
	where $r^2= \bar{x}^2 + \bar{y}^2$, $\eta = \frac{r}{r_c}$, and $(\bar{x}, \bar{y}) = (x-x_c, y-y_c)$. $\varepsilon = 0.3$ is the strength of the vortex, $\alpha = 0.204$ denotes the decay rate of the vortex, and $r_c=0.05$ represents the critical radius of the vortex. 
	The inflow boundary conditions are applied on the left boundary and the outflow boundary conditions are imposed on the right boundary. Meanwhile, reflective boundary conditions are applied on both the upper and lower boundaries.
	
	We divide $\Omega$ into $400\times 200$ uniform rectangular cells and conduct the simulation using the 2D $\mathbb{Q}^2$-based OESV scheme up to $t=0.8$. \cref{fig:vortex} presents the pressure contour plots of the numerical solution at six different time instances. The results obtained by the proposed OESV scheme match well with those reported in \cite{Liu2022AnEO}, without producing any nonphysical oscillations.

	\begin{figure}[!htb]
		\centering
		\begin{subfigure}[h]{.32\linewidth}
			\centering 
			\includegraphics[width=0.99\textwidth]{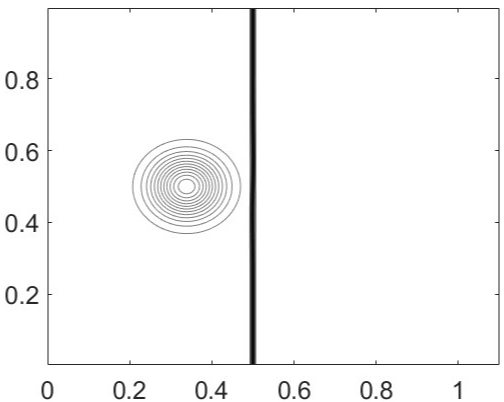}
			\subcaption{$t=0.068$}
		\end{subfigure}
		\hfill 
		\begin{subfigure}[h]{.32\linewidth}
			\centering 
			\includegraphics[width=0.99\textwidth]{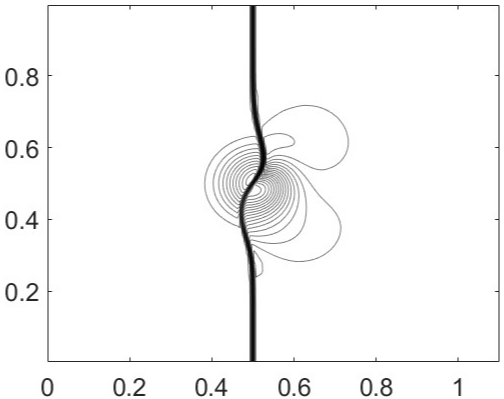}
			\subcaption{$t=0.203$}
		\end{subfigure}
		\hfill 
		\begin{subfigure}[h]{.32\linewidth}
			\centering 
			\includegraphics[width=0.99\textwidth]{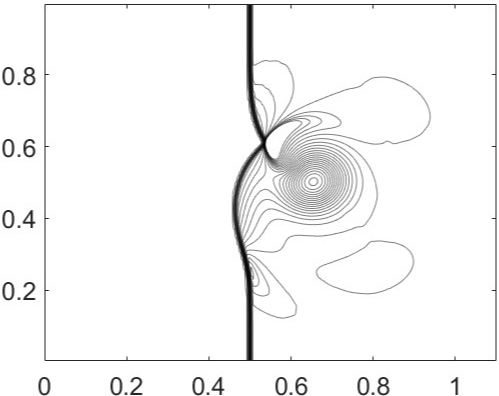}
			\subcaption{$t=0.330$}
			
		\end{subfigure}
		\vspace{3.6mm}
		
		\begin{subfigure}[h]{.32\linewidth}
			\centering 
			\includegraphics[width=0.99\textwidth]{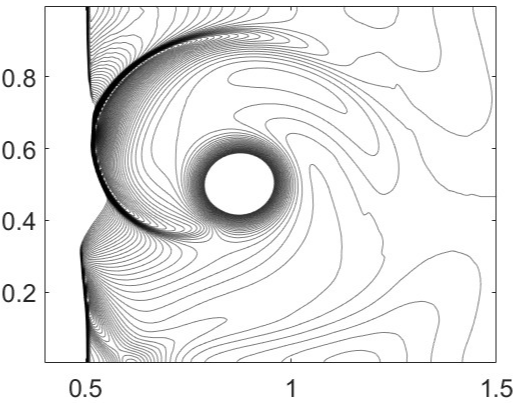}
			\subcaption{$t=0.529$}
		\end{subfigure}
		\hfill 
		\begin{subfigure}[h]{.32\linewidth}
			\centering 
			\includegraphics[width=0.99\textwidth]{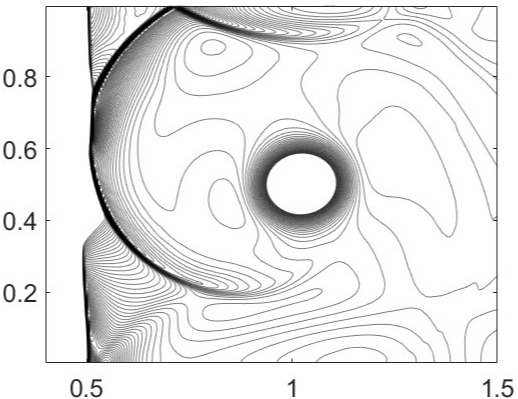}
			\subcaption{$t=0.662$}
		\end{subfigure}
		\hfill 
		\begin{subfigure}[h]{.32\linewidth}
			\centering 
			\includegraphics[width=0.99\textwidth]{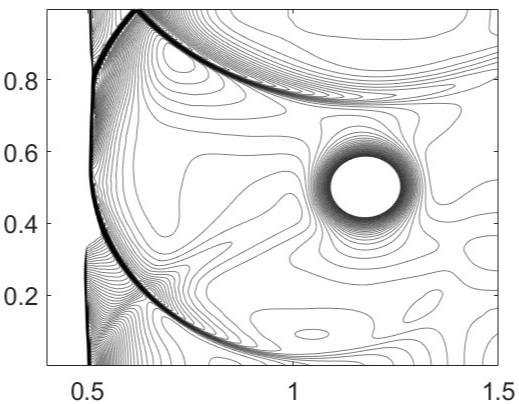}
			\subcaption{$t=0.8$}
		\end{subfigure}
		
		\caption{Contour plots of pressure for shock-vortex interaction:  
			50 contour lines from 0.68 to 1.3 are shown in the top figures, while 90 contour lines from 1.19 to 1.36 are plotted in the bottom figures.}
		\label{fig:vortex}
	\end{figure}
\end{exmp}

\subsection{Proof of \cref{prop:Mstar-surjective}}\label{app:1}
If $k=0$, then $\mathbb{V}^{k,*} = \mathbb{V}^{k}$ and $M^{*}\omega = \omega$ for all $ \omega \in \mathbb{V}^{k}$. If $k\geq 1$, then for every $\omega^* \in \mathbb{V}^{k,*}$, we define $\omega_{i,j}^*$ as the constant value of $\omega^*$ on $I_{i,j}$. Then, we can find $ p_i \in \mathbb{P}^{k-1}(I_{i})$ for each $i$ such that
$$ p_i(x_{i,j}) = \begin{cases}
	0,~~ & \mbox{if}~ A_{i,j}=0, \\
	\frac{\omega_{i,j}^*-\omega_{i,j-1}^*}{A_{i,j}},~~ & \mbox{otherwise},
\end{cases}\quad j =1,\dots,k. $$
Note that there exists $ P_i \in \mathbb{P}^{k}(I_{i})$ such that $(P_i)_x = p_i$ and $P_i(x_{i-\frac{1}{2}}) = \omega_{i,0}^* - A_{i,0}p_i(x_{i-\frac{1}{2}})$. Hence, if we define $\omega \in \mathbb{V}^{k}$ by $\omega |_{I_i} = P_i$ for all $i$, then $\omega^* = M^{*}\omega$. The proof is completed.

\subsection{Proof of \cref{prop:Mstar}}\label{app:2}
For any $\omega \in \mathbb{V}^{k}$,  the Cauchy--Schwarz inequality and the definition of $Q_i^k$ imply
\begin{equation}\label{eq:ieq-Mstar}
	\nm{M^*\omega}^2 \leq (k+1) \sum_i \left( h_i\nm{\omega}_{L^\infty (I_i)}^2 + h_i^2Q_i^k(\omega_x^2) \right).
\end{equation}
Since $\omega_x^2 |_{I_i} \in \mathbb{P}^{2k-1}(I_i)$, we have $Q_i^k(\omega_x^2) = \nm{\omega_x}_{L^2 (I_i)}^2$. Applying inverse inequalities \eqref{eq:normequiv-1} and \eqref{eq:normequiv} to the right hand side of \eqref{eq:ieq-Mstar} completes the proof.

\section{Proof of \cref{prop:ipstar}}\label{app:3}
\cref{lemma:Mstar-difference} implies $\ip{v}{\omega}_* = \ip{v}{\omega} + \sum_i R_i^k(V\omega_x)$ for some $V \in \mathbb{V}^{k+1}$. For $\omega \in \mathbb{V}^{k-1}$, $R_i^k(V\omega_x)=0$ for all $i$, which implies \eqref{eq:propipstar_1}. To prove \eqref{eq:propipstar_2}, we take $C = 1+\sqrt{\nm{M^*}}$, and \cref{prop:Mstar} indicates that
\begin{equation}
	\nm{v}_*^2 = \ip{v}{M^* v} \leq \nm{M^*} \nm{v}^2 \leq \left( C \nm{v} \right)^2.
\end{equation}
Using \eqref{eq:normLik} and \eqref{eq:propipstar_1}, we can verify that $c\nm{v} \leq \nm{v}_*$ for any $v \in \mathbb{V}^k$, where
\begin{equation}
	c = \min \left\{ 1, \min_i \left\{ 1-(2k-1)\frac{Q_i^k(L_{i,k+1}L_{i,k-1})}{h_i} \right\} \right\}.
\end{equation}
The constant $c>0$ is independent of the mesh $\{ I_i \}$, since $\frac{Q_i^k}{h_i}$ is independent of $h_i$ (\cite{Brass2011QuadratureTT}).

\subsection{Proof of the statement in \cref{rmk:stability-OESV}}\label{sec:proof11}
Without the loss of generality, we assume that $g \equiv 0$ in \eqref{eq:1Dlinear}. First, we observe from \eqref{eq:filter} and \eqref{eq:propipstar_1} that
\begin{equation}\label{eq:05192100}
	\nm{\mathcal{F}_\tau v}_* \leq \nm{v}_*\quad \forall v \in \mathbb{V}^k.
\end{equation}
We will prove the statement in \cref{rmk:stability-OESV} by induction for $\nm{u_h^{n,\ell+1}}_*$.

For the OESV scheme \eqref{eq:linearOESV}, using \eqref{eq:SV-td} gives $u_h^{n,1} = \mathcal{F}_\tau \left( u_h^n + d_{00}\svtd u_h^n \right)$. Because $d_{00} \geq 0$, \cref{prop:SVtd-prop} and \eqref{eq:05192100} yield that
\begin{equation}\label{eq:05192115}
	\begin{aligned}
		\nm{u_h^{n,1}}_*^2 & \leq \ip{u_h^n + d_{00}\svtd u_h^n}{u_h^n + d_{00}\svtd u_h^n}_* \\
		& \leq \left( 1 + d_{00}^2 \nm{\svtd}_*^2 \right)\nm{u_h^{n}}_*^2 - d_{00}\jump{u_h^{n}}^2 \\
		& \leq \left( 1 + C \left( \frac{\tau}{h} \right)^2 \right)\nm{u_h^{n}}_*^2 \leq \left( 1+C_{(0)}\tau \right)\nm{u_h^{n}}_*^2,
	\end{aligned}
\end{equation}
for some $C_{(0)} \geq 0$ if $\frac{\tau}{h^2} \leq C_{\mathrm{CFL}}^* $. Now suppose that $\nm{u_h^{n,\ell+1}}_*^2 \leq \left( 1+C_{(\ell)}\tau \right)\nm{u_h^{n}}_*^2$ for $0\leq \ell< \ell_0$ when $\frac{\tau}{h^2} \leq C_{\mathrm{CFL}}^* $. Because $d_{\ell_0 \kappa}\geq 0$, similar to the derivation of \eqref{eq:05192115}, we can show that if $\frac{\tau}{h^2} \leq C_{\mathrm{CFL}}^* $, then 
\begin{equation}\label{eq:05192210}
	\nm{u_h^{n,\kappa} + \frac{d_{\ell_0 \kappa}}{c_{\ell_0 \kappa}}\svtd u_h^{n,\kappa}}_*^2 \leq \left( 1+C\tau \right)\nm{u_h^{n,\kappa}}_*^2 \leq \left( 1+C\tau \right)\nm{u_h^{n}}_*^2,
\end{equation}
for $0 \leq \kappa \leq \ell_0$. Here, we set $\frac{d_{\ell_0 \kappa}}{c_{\ell_0 \kappa}} = 1$ if $c_{\ell_0 \kappa} = 0$. Notice that \eqref{eq:linearOESV} implies 
\begin{equation*}
	u_h^{n,\ell_0 + 1} = \sum_{0 \leq \kappa \leq \ell_0 } c_{\ell_0 \kappa}\left( u_h^{n,\kappa} + \frac{d_{\ell_0 \kappa}}{c_{\ell_0 \kappa}}\svtd u_h^{n,\kappa} \right).
\end{equation*}
As $c_{\ell_0 \kappa} \geq 0$ and $\sum_\kappa c_{\ell_0 \kappa}=1$,  the convexity of the function $h(x)=x^2$  yields 
\begin{equation}\label{eq:05192220}
	\begin{aligned}
		\nm{u_h^{n,\ell_0 + 1}}_*^2 & \leq \left( \sum_{0 \leq \kappa \leq \ell_0 } c_{\ell_0 \kappa}\nm{u_h^{n,\kappa} + \frac{d_{\ell_0 \kappa}}{c_{\ell_0 \kappa}}\svtd u_h^{n,\kappa}}_* \right)^2 \\
		& \leq \sum_{0 \leq \kappa \leq \ell_0 } c_{\ell_0 \kappa}\nm{u_h^{n,\kappa} + \frac{d_{\ell_0 \kappa}}{c_{\ell_0 \kappa}}\svtd u_h^{n,\kappa}}_*^2 \leq \left( 1+C_{(\ell_0)}\tau \right)\nm{u_h^{n}}_*^2,
	\end{aligned}
\end{equation}
when $\frac{\tau}{h^2} \leq C_{\mathrm{CFL}}^* $. By \eqref{eq:05192115} and \eqref{eq:05192220}, the desired result \eqref{eq:stability-RKSV} for the OESV scheme \eqref{eq:linearOESV} is then verified using mathematical induction.

\subsection{Proof of \cref{prop:OE-1}}
According to \cite{Peng2023OEDGOD},  we introduce the following notations:
\begin{equation*}
	\tilde{\xi}^{n, \ell+1} = \tilde{u}_h^{n,\ell+1} - \mathbf{P} U^{n,\ell+1},\quad \eta^{n, \ell+1} = U^{n,\ell+1} - \mathbf{P} U^{n,\ell+1}
\end{equation*}
where $\mathbf{P}$ denotes the Gauss--Radau projection.  \cite[Section 4.4.2]{Peng2023OEDGOD} has shown that there exists a constant $h_0>0$ independent of $n$ and $\ell$ such that
\begin{equation}\label{eq:05151713}
	\nm{\tilde{\xi}^{n,\ell+1}}_{L^\infty (\Omega)} \leq 2h \leq 2 h_0 ~ \Rightarrow ~ \nm{\cF^{n,\ell+1}} \leq C\left( \nm{\tilde{\xi}^{n,\ell+1}} + h^{k+1} \right).
\end{equation}
By the approximation property of $\mathbf{P}$ and $P^*$, we have $\nm{\eta^{n, \ell+1}} \leq C h^{k+1}$ and $\nm{\eta^{n, \ell+1,*}} \leq C h^{k+1}$. Applying the inverse inequality \eqref{eq:normequiv}, we then obtain
\begin{equation*}
	\begin{aligned}
		\nm{\tilde{\xi}^{n,\ell+1}}_{L^\infty (\Omega)} & \leq \nm{\tilde{\xi}^{n,\kappa,*}}_{L^\infty (\Omega)} + \nm{\eta^{n, \ell+1,*}}_{L^\infty (\Omega)} + \nm{\eta^{n, \ell+1}}_{L^\infty (\Omega)} \\
		& \leq \nm{\tilde{\xi}^{n,\kappa,*}}_{L^\infty (\Omega)} + Ch^{k+\hf},
	\end{aligned}
\end{equation*}
where the constant $C>0$ is independent of $n$ and $\ell$. Since we assume that $k\ge 1>\hf$, there exists $h_*\in (0,h_0]$ such that $Ch^{k+\hf} \leq h$ when $h \leq h_*$. If $\nm{\tilde{\xi}^{n,\ell+1,*}}_{L^\infty (\Omega)} \leq h \leq h_*$, \eqref{eq:05151713} and \eqref{eq:propipstar_2} imply
\begin{equation}
	\begin{aligned}
		\nm{\cF^{n,\ell+1}} & \leq C\left( \nm{\tilde{\xi}^{n,\ell+1,*}} + \nm{\eta^{n, \ell+1,*}} + \nm{\eta^{n, \ell+1}} + h^{k+1} \right) \\
		& \leq C\left( \nm{\tilde{\xi}^{n,\ell+1,*}}_* + h^{k+1} \right).
	\end{aligned}
\end{equation}
The proof is completed.

\subsection{Proof of \cref{prop:xi-tilde}}
By \eqref{eq:linearOESV}, one can verify that $\{ \tilde{\xi}^{n,\mykappa,*} \} $ satisfies the following RKSV scheme:
\begin{equation*}\label{eq:05151725}
	\ip{\tilde{\xi}^{n,\ell+1,*}}{\omega}_* = \sum_{0 \leq \mykappa \leq \ell} \left(c_{\ell \mykappa}\ip{\xi^{n,\mykappa,*}}{\omega}_* + \tau d_{\ell \mykappa}\sv{\xi^{n,\mykappa,*}}{\omega}\right) + \tau\ip{\cZ^{n,\ell+1,*}}{M^* \omega}.
\end{equation*}
If we take $\omega = \tilde{\xi}^{n,\ell+1,*}$, then using \cref{prop:Mstar}, \eqref{eq:propipstar_2}, \eqref{eq:inequalityH}, \eqref{eq:05151721}, and the Cauchy-Schwarz inequality, gives 
\begin{equation}\label{eq:07052024-2}
	\nm{\tilde{\xi}^{n,\ell+1,*}}_*^2 \leq C\nm{\tilde{\xi}^{n,\ell+1,*}}_* \left( C \bigg( \sum_{0 \leq \kappa \leq \ell} \nm{\xi^{n,\kappa,*}}_*  \bigg) + C\tau(h^{k+1}+\tau^r) \right).
\end{equation}
This completes the proof.

\bibliography{refs}
\bibliographystyle{siamplain}

\end{document}